\documentclass[a4paper,12pt,reqno]{amsart}
\usepackage{enumerate}
\usepackage{lmodern}
\usepackage{amssymb}
\usepackage{mathrsfs}

\newtheorem{lemma}{Lemma}
\newtheorem{remark}{Remark}

\newtheorem{theorem}{Theorem}

\newtheorem{definition}{Definition}

\newcommand{\vek}[1]{\mathbf{#1}}

\newcommand{\abs}[1]{\lvert#1\rvert}
\newcommand{\mat}[1]{\mathbf{#1}}

\newcommand{\gauss}[3]{\genfrac{[}{]}{0pt}{}{#1}{#2}_{#3}}

\DeclareMathOperator{\PG}{PG}

\DeclareMathOperator{\GL}{GL}

\DeclareMathOperator{\AGL}{AGL}

\DeclareMathOperator{\Hom}{Hom}
\DeclareMathOperator{\End}{End}

\DeclareMathOperator{\rank}{rk}
\DeclareMathOperator{\identity}{id}
\DeclareMathOperator{\kernel}{Ker}
\DeclareMathOperator{\image}{Im}

\newcommand{\imat}{\mathbf{I}}

\newcommand{\trace}{\mathrm{Tr}}

\newcommand{\F}{\mathbb{F}}
\DeclareMathOperator{\GF}{GF}

\newcommand{\Q}{\mathbb{Q}}

\newcommand{\N}{\mathbb{N}}
\newcommand{\Z}{\mathbb{Z}}

\newcommand{\steiner}{\mathrm{S}}
\newcommand{\dickson}{\delta}
\newcommand{\sickson}{\sigma}
\newcommand{\graph}{\Gamma}
\newcommand{\gabidulin}{\mathcal{G}}
\newcommand{\rspace}{\mathcal{R}}
\newcommand{\tspace}{\mathcal{T}}
\newcommand{\dspace}{\mathcal{D}}

\newcommand{\cpol}{\Phi}
\newcommand{\frob}{\varphi}
\newcommand{\uone}{\epsilon}
\newcommand{\utwo}{\varepsilon}

\newcommand{\thomas}{\mathcal{T}}
\DeclareMathOperator{\cm}{cm}

\title[On Putative $q$-Analogues of the Fano Plane]
  {On putative $q$-Analogues of the Fano Plane and Related
  Combinatorial Structures}

\author{Thomas Honold}
\address{Department of Information Science and Electronics Engineering,
  Zhejiang University, 38 Zheda Road,
  310027 Hangzhou, China}
\email{honold@zju.edu.cn}

\author{Michael Kiermaier}
\address{Mathematisches Institut,
  Universit\"at Bayreuth,
  D-95440 Bayreuth,
  Germany}
\email{michael.kiermaier@uni-bayreuth.de} 

\date{April 20, 2015}

\dedicatory{Herrn Professor Armin Leutbecher zum 80.\ Geburtstag}

\begin{document}

\begin{abstract}
  A set $\mathcal{F}_q$ of $3$-dimensional subspaces of $\F_q^7$, the
  $7$-dimensional vector space over the finite field $\F_q$, is said
  to form a $q$-analogue of the Fano plane if every $2$-dimensional
  subspace of $\F_q^7$ is contained in precisely one member of
  $\mathcal{F}_q$. The existence problem for such $q$-analogues
  remains unsolved for every single value of $q$. Here we report on an
  attempt to construct such $q$-analogues using ideas from the theory
  of subspace codes, which were introduced a few years ago by Koetter
  and Kschischang in their seminal work on error-correction for network
  coding. Our attempt eventually fails, but it produces the largest
  subspace codes known so far with the same parameters as a putative
  $q$-analogue. In particular we find a ternary subspace code of new
  record size $6977$, and we are able to construct a binary subspace
  code of the largest currently known size $329$ in an entirely
  computer-free manner.
\end{abstract}

\maketitle

\section{Introduction}\label{sec:intro}

The Fano plane $\mathcal{F}=\PG(2,\F_2)=\PG(\F_2^3/\F_2)$, the
coordinate geometry derived from a $3$-dimensional vector space over
the binary field $\F_2$, is the smallest nontrivial model of an abstract
projective geometry. It has $7$ points and $7$ lines, represented by
the one- and two-dimensional subspaces of $\F_2^3/\F_2$, respectively;
each line contains $3$ points and each point is on $3$ lines; any two
distinct points are contained in a unique line and any two
distinct lines intersect in a unique point. Myriads of other
finite models of a projective geometry 
exist---for each integer $n\geq 2$ and prime power $q>1$ the
$n$-dimensional coordinate geometry $\PG(n,\F_q)=\PG(\F_q^{n+1}/\F_q)$
over the finite field $\F_q$, and in the planar case many additional
examples with the same parameters as some $\PG(2,\F_q)$.

% Actual realizations are
% provided by taking the point set as $\Z_7=\Z/7\Z$ and the lines as
% $D=\{1,2,4\}$ and its additive shifts $D+i$, $i\in\Z_7$, or as the
% coordinate geometry $\PG(V/\F_2)$ of a $3$-dimensional vector space
% over $\F_2$. Taking in particular $V=\F_8=\F_2[\alpha]$ subject to
% $\alpha^3+\alpha+1=0$ and identifying, as usual, the points of
% $\PG(\F_8/\F_2)$ (i.e., the $1$-dimensional subspaces of $V/\F_2$)
% with the elements in $\F_8^\times=\F_8\setminus\{0\}$ and the lines
% ($2$-dimensional subspaces of $V/\F_2$) with sets of points, we see
% that $\Z_7\to\F_8^\times$, $i\mapsto\alpha^i$ provides an
% identification between both views.

The Fano plane $\mathcal{F}=\steiner(2,3,7)$ is also the smallest
nontrivial example of a \emph{Steiner system} $\steiner(t,k,v)$, which
refers to a $v$-set $V$ (\emph{point set}) and a set of $k$-subsets of
$V$ (\emph{blocks}) having the property that any $t$-subset of $V$ is
contained in exactly one block. The more general concept of a
combinatorial $t$-$(v,k,\lambda)$ design relaxes the requirement
``exactly one block'' to a ``constant number $\lambda$ of
blocks''. Many constructions of $t$-$(v,k,\lambda)$ designs are known
(including the construction of nontrivial $t$-designs for all positive
integers $t$ by Teirlinck \cite{teirlinck87}), but comparatively few
Steiner systems and still no one at all with $t>5$).\footnote{We
  should note here that recently Keevash \cite{keevash14} has given a
  non-constructive proof of the existence of Steiner systems for all
  values of $t$.}

This article is concerned with vector space analogues of $\mathcal{F}$
in the following sense:
\begin{definition}
  \label{dfn:main}
  Let $q>1$ be a prime power. A set $\mathcal{F}_q$
  of $3$-dimensional subspaces of $\F_q^7/\F_q$ (or any other
  $7$-dimensional vector space $V$ over $\F_q$) is said to be a
  \emph{$q$-analogue of the Fano plane} if every $2$-dimensional
  subspace of $\F_q^7$ (respectively, $V$) is contained in a unique
  member of $\mathcal{F}_q$. 
\end{definition}
In projective geometry language, a $q$-analogue of the Fano plane is a
set $\mathcal{F}_q$ of planes in $\PG(6,\F_q)$ such that any pair of
distinct points (equivalently, any line) is contained in exactly one
plane $E\in\mathcal{F}_q$. In other words, the planes in
$\mathcal{F}_q$, when identified with sets of lines, should form an
exact cover (i.e., a partition) of the line set of $\PG(6,\F_q)$.

Before going any further, we should remark that at the time of
writing this article virtually nothing is known about the existence
of such structures---neither existence nor non-existence of a
$q$-analogue of the Fano plane has been proved for a single instance
of $q$. Even in the smallest case $q=2$, where a putative $2$-analogue
$\mathcal{F}_2$ would have to contain $381$ of the $11811$ planes of
$\PG(6,\F_2)$, a computer search seems infeasible at present.

P.~Cameron \cite{cameron74} introduced the concept of a \emph{design
  over a finite field} as a vector space analogue (``$q$-analogue'',
if the underlying field is $\F_q$) of combinatorial designs: A
$t$-$(v,k,\lambda)$ design over $\F_q$ is a set $\mathcal{C}$ of
$k$-dimensional subspaces of $\F_q^v/\F_q$ (or any other
$v$-dimensional vector space $V$ over $\F_q$) with the property that every
$t$-dimensional subspace of $\F_q^v$ (respectively, $V$) is contained
in exactly $\lambda$ members of $\mathcal{C}$.  The first nontrivial
examples of such designs were constructed by S.~Thomas
\cite{thomas87}. These ``Thomas designs'' have $q=2$ and form an
infinite family with parameters $2$-$(v,3,7)$, where
$v\equiv\pm1\pmod{6}$ and $v\geq 7$.  Taking the ambient space as the
finite field $\F_{2^v}$, one may construct the $2$-$(v,3,7)$ Thomas
design $\thomas_v$ as the set of all $3$-dimensional $\F_2$-subspaces
$\langle x,y,z\rangle\subset\F_{2^v}$ spanned by the $2^v-2$
non-rational points in $\PG(2,\F_{2^v})$ of a rational conic (relative
to $\F_2$). For example, we can take all points $(x:y:z)\neq(1:0:0)$,
$(0:1:0)$, $(0:0:1)$ on the conic $xy+yz+zx=0$, resulting in
$\thomas_v=\bigl\{\langle
x,y,\frac{xy}{x+y}\rangle;x,y\in\F_{2^v}^\times\text{
  distinct}\bigr\}$.\footnote{Checking
  the design property is somewhat tedious, but at least we can see
  immediately from the definition that $\thomas_v$ has the required
  $(2^v-2)/6\times(2^v-1)=(2^v-1)(2^{v-1}-1)/3$ blocks.}  Although
several further constructions of designs over finite fields are now
known (including the existence of nontrivial $t$-designs over $\F_q$
for arbitrarily large $t$ in \cite{fazeli-lovett-vardy13}), the
subject has turned out considerably more difficult than ordinary
combinatorial design theory. For example, no nontrivial $4$-design
over a finite field is known at present.

At the end of \cite{thomas87} Thomas briefly discussed $q$-analogues
$\steiner_q(t,k,v)$ of Steiner systems (i.e.\ $t$-$(v,k,1)$ designs
over $\F_q$) and in particular the smallest feasible parameter case
$\steiner_2(2,3,7)$. Such a $2$-analogue of the Fano plane would
consist of $381=3\times 127$ three-dimensional subspaces of $\F_2^7$
(cf.\ Lemma~\ref{lma:q-analog}), and it was conceivable to construct it as
the union of $3$ orbits of a Singer subgroup of $\GL(2,7)$. However,
as Thomas reported, this construction is impossible.

A few years ago interest in designs over finite fields was revived
through the observation by R.~Koetter and F.~Kschischang
\cite{koetter-kschischang08} that sets of subspaces of a vector space
over a finite field (\emph{subspace codes}) can be used as
``distributed channel codes'' for error-resilient transmission of
information in packet networks. Considering $q$ (\emph{symbol alphabet
  of the packet network}) and the ambient vector space dimension $v$
(\emph{packet length}) as fixed and restricting attention to
\emph{constant-dimension} codes (i.e\ the dimension $k$ of all
codewords is the same), the best performance is achieved by using
subspace codes $\mathcal{C}$ that have simultaneously large size
$\#\mathcal{C}=\abs{\mathcal{C}}$ and small maximum dimension of an
intersection between distinct codewords. Denoting this dimension by
$t-1$, we have that $t$ is the smallest positive integer such that every
$t$-dimensional subspace of $\F_q^v$ is contained in at most one
codeword of $\mathcal{C}$. Subspace codes thus satisfy a weaker form
of the defining condition for Steiner systems over finite
fields.\footnote{The difference is quite similar to that between
  \emph{linear spaces} (two distinct points are connected by exactly
  one line) and \emph{partial linear spaces} (two distinct points are
  connected by at most one line), as defined in Incidence
  Geometry. Subspace codes could thus be called ``partial Steiner
  systems over finite fields''.}  A standard double-counting argument
gives $\#\mathcal{C}\times\gauss{k}{t}{q}\leq\gauss{v}{t}{q}$ with
equality if and only if each $t$-dimensional subspace of $\F_q^v$ is
contained in precisely one codeword of $\mathcal{C}$. Hence Steiner
systems over finite fields are optimal as subspace
codes.\footnote{From this we also see that the parameters $q$, $t$,
  $k$, $v$ of an $\steiner_q(t,k,v)$, like those of an ordinary
  Steiner system, must obey certain \emph{integrality conditions}. In
  fact the existence of an $\steiner_q(t,k,v)$ implies the existence
  of an $\steiner_q(t-1,k-1,v-1)$. The so-called \emph{derived
    designs}, which are formed by the blocks through a fixed
  $1$-dimensional subspace of $\F_q^v$, have these parameters. Hence
  a necessary condition for the existence of an $\steiner_q(t,k,v)$ is
  that $\gauss{v-s}{t-s}{q}/\gauss{k-s}{t-s}{q}$ must be an integer
  for $1\leq s\leq t$.}

In the sequel we will exclusively be concerned with subspace codes of
constant dimension $k=3$, so-called \emph{plane subspace codes}, and
packet length $v=7$. Plane subspace codes with $t=3$ are trivial---the
whole plane set of $\PG(6,\F_q)$ forms such a code. Plane subspace
codes with $t=1$ consist of pairwise skew planes and are known as
\emph{partial plane spreads} in Finite Geometry. The maximum size of a
partial plane spread in $\PG(6,\F_q)$ is known to be $q^4+1$ from the
work of Beutelspacher \cite[Th.~4.1]{beutelspacher75}.\footnote{In
  general, the maximum size of a partial plane spread in $\PG(v-1,\F_q)$
  is known for $v\equiv 0,1\bmod{3}$ (all $q$) and for $q=2$ (all
  $v$); for the latter see
  \cite{el-zanati-jordon-seelinger-sissokho-spence10}.}
  This leaves the case $t=2$ considered so far as the only unresolved
  case. Restricting attention to this case, we will from now on
  tacitly assume that ``subspace code'' includes the assumption $t=2$.   

More than $25$ years have passed since Thomas' fundamental work
\cite{thomas87} and the existence problem for $q$-analogues
of the Fano plane is still undecided. 
% Moreover, the
% $q$-analogues $\steiner_q(2,3,7)$ all have feasible parameter sets and
% their existence remains undecided as well. 
On the other hand, serious attempts, often relying on quite
sophisticated computational methods, have been made to construct large
subspace codes---including the parameter set of a putative
$2$-analogue. These will now be briefly reviewed. Accordingly, let
$\mathcal{C}$ be a binary plane subspace code with $v=7$ or, in
geometric terms, a set of planes in $\PG(6,\F_2)$ mutually
intersecting in at most a point. As discussed above, we have
$\#\mathcal{C}\leq381$ with equality if and only if $\mathcal{C}$ is a
$2$-analogue of the Fano plane. The first nontrivial lower bound on
the maximum size of $\mathcal{C}$ was established by Koetter and
Kschischang \cite{koetter-kschischang08}, who showed that
$\#\mathcal{C}=256$ is realized by a so-called lifted maximum-rank
distance code (LMRD code). Kohnert and Kurz \cite{kohnert-kurz08}
improved this to $\#\mathcal{C}=304$, employing a computer search for
plane subspace codes in $\PG(6,\F_2)$ with an automorphism of order $21$
acting irreducibly on a hyperplane.  The current record is
$\#\mathcal{C}=329$ and was established by Braun and Reichelt in
\cite{braun-reichelt14} using a refinement of this method. In
\cite{smt:fq11proc}, as part of the classification of all optimal
plane subspace codes in the smaller geometry $\PG(5,\F_2)$, an optimal
$\#\mathcal{C}=77$ subspace code was constructed by first expurgating
an LMRD code (size $64$) to a particular subspace code of size $56$
and then augmenting this code by $21$ further planes. As shown in
\cite{lt:0328}, the underlying idea can be used to provide an
alternative construction of a plane subspace code of size $329$ in
$\PG(6,\F_2)$.

In this paper we will develop a general framework for constructing
large plane subspace codes in $\PG(6,\F_q)$ along the lines of
\cite{smt:fq11proc,lt:0328}, but also introducing several new ideas (in
Sections~\ref{sec:aelmrd} and~\ref{sec:attempt}). Our main results are
the construction of a general $q$-ary subspace code $\mathcal{C}$ of
size $q^8+q^5+q^4-q-1$, whose planes meet a fixed solid ($3$-flat) of
$\PG(6,\F_q)$ in at most a point (Theorem~\ref{thm:sickson} in
Section~\ref{sec:attempt}), and a detailed analysis of the extension
problem for $\mathcal{C}$ (or rather, a distinguished subcode
$\mathcal{C}_0\subset\mathcal{C}$) by planes meeting $S$ in a line,
which enables us to give the first computer-free construction of a
plane subspace code of size $329$ in $\PG(6,\F_2)$ (see above) and a
computer-aided construction of a plane subspace code of size $6977$ in
$\PG(6,\F_3)$ (Section~\ref{sec:ext}, in particular
Theorem~\ref{thm:final}). Theorem~\ref{thm:sickson} improves the best
previously known construction for general $q$
\cite{trautmann-rosenthal10}, and the ternary subspace code of size
$6977$ is by way the largest known code with its parameters. In order
to make the paper self-contained, we provide a general introduction to
the combinatorics of subspace codes in Section~\ref{sec:counting} and
an account of related previous subspace code constructions in
Section~\ref{sec:almrd}.

In the sequel $\mathcal{F}_q$ always denotes a putative $q$-analogue of the
Fano plane. The term ``dimension'' refers to vector space
dimension, but otherwise geometric language will be extensively used. 
When referring to the geometric dimension of a $t$-dimensional
subspace of $\F_q^v/\F_q$, we use the term ``\emph{($t-1$)-flat of
  $\PG(v-1,\F_q)$}''.

Let us close this introduction with a remark on vector space analogues of
the Fano plane over infinite fields. Using transfinite recursion, it
is fairly easy to show that for any field $K$ with $\abs{K}=\infty$ a
$K$-analogue $\mathcal{F}_K$, defined as in Def.~\ref{dfn:main}, does
exist. For example, in the case $K=\Q$ we can enumerate the lines of
$\PG(6,\Q)$ as $L_0,L_1,L_2,\dots$ and recursively define sets
$\mathcal{E}_0=\emptyset$, $\mathcal{E}_1,\mathcal{E}_2,\dots$ of
planes as follows: If $\mathcal{E}_n$ already contains a plane
$E\supset L_n$, we set $\mathcal{E}_{n+1}=\mathcal{E}_n$; otherwise,
among the planes containing $L_n$ there exists a plane $E$ that has no
line in common with any of the planes in $\mathcal{E}_n$, and we set
$\mathcal{E}_{n+1}=\mathcal{E}_n\cup\{E\}$.\footnote{More precisely,
  the first plane with this property, according to some predefined
  order $E_0,E_1,E_2,\dots$ on the set of planes of $\PG(6,\Q)$, is
  chosen. The existence of such a plane follows from the fact that
  $L_n$ and the finitely many solids ($3$-flats) $E'+L_n$,
  $E'\in\mathcal{E}_n$ a plane intersecting $L_n$ (necessarily in a
  point), cannot cover all points of $\PG(6,\Q)$.}  It is then readily
verified that $\mathcal{F}_\Q=\bigcup_{n=0}^\infty\mathcal{E}_n$ is
the required $\Q$-analogue of $\mathcal{F}$. 

In fact it is even true that the plane set
of any geometry $\PG(v-1,K)$, $\abs{K}=\infty$, $v\geq 5$, can be
partitioned into Steiner systems $\steiner_K(2,3,v)$; see
\cite{cameron95c} for details.

\section{Counting Preliminaries}\label{sec:counting}

Let us first recall that the number of $k$-dimensional subspaces of an
$n$-dimensional vector space over $\F_q$ is given by the Gaussian
binomial coefficient
\begin{equation*}
  \gauss{n}{k}{q}=\frac{(q^n-1)(q^{n-1}-1)\dotsm(q^{n-k+1}-1)}{(q^k-1)(q^{k-1}-1)
    \dotsm(q-1)},
\end{equation*}
which is polynomial in $q$ of degree $k(n-k)$ and satisfies
$\gauss{n}{k}{q}=\gauss{n}{n-k}{q}=q^{k(n-k)}\cdot\gauss{n}{k}{q^{-1}}$.
In particular the number of points (and hyperplanes) of
$\PG(n-1,\F_q)$ is equal to
\begin{equation*}
  \gauss{n}{1}{q}=\gauss{n}{n-1}{q}=\frac{q^n-1}{q-1}=1+q+\dots+q^{n-1}.
\end{equation*}
Subspaces $U$ of $\F_q^n/\F_q$ of dimension $k$ are in one-to-one
correspondence with matrices $\mat{U}=\cm(U)\in\F_q^{k\times n}$ in
reduced row-echelon form via $U=\langle\cm(U)\rangle$, the row space
of the matrix $\cm(U)$, and
$\mat{U}=\cm(\langle\mat{U}\rangle)$.\footnote{The name 'cm'
  resembles ``canonical matrix''.} If $\cm(U)$ has pivot columns in
positions $1\leq j_1<j_2<\dots<j_k\leq n$ then the number of
unspecified entries (``wildcards'') in $\cm(U)$ is
$i=1(j_2-j_1-1)+2(j_3-j_2-1)+\dots+(k-1)(j_k-j_{k-1}-1)+k(n-j_k)$ and
determines a partition of the integer $i$ into at most $n-k$ parts of
size at most $k$.\footnote{The number of (positive) parts is
  $\sum_{\nu=1}^{k-1}(j_{\nu+1}-j_\nu-1)+n-j_k=n-k-(j_1-1)$.} The
coefficient $a_i$ of $q^i$ in $\gauss{n}{k}{q}$ counts the number of
such partitions, and consequently the monomial $a_iq^i$ counts the
$k$-dimensional subspaces of $\F_q^n$ having exactly $i$ unspecified
entries in their canonical matrix.

These and a few additional observations allow for ``almost
everything'' in $\PG(n-1,\F_q)=\PG(\F_q^n/\F_q)$ to be counted. Consider,
for example, any solid ($3$-flat) $S$ in $\PG(6,\F_q)$ and count
the planes of $\PG(6,\F_q)$ according to their intersection size with
$S$.  There are $q^{12}$ planes disjoint from $S$, corresponding to
the $q^{12}$ canonical matrices
\begin{equation*}
  \begin{pmatrix}
    1&0&0&*&*&*&*\\
    0&1&0&*&*&*&*\\
    0&0&1&*&*&*&*
  \end{pmatrix}
\end{equation*}
(for this arrange coordinates such that $S=(0,0,0,*,*,*,*)$); there
are $q^6\cdot\gauss{3}{2}{q}\gauss{4}{1}{q}=q^6(q^2+q+1)
(q^3+q^2+q+1)$ planes $E$ meeting $S$ in a point (considering the
hyperplane $H=E+S$ and the intersection point $P=E\cap S$ as fixed,
these correspond to lines disjoint from the plane $S/P$ in
$H/P\cong\PG(4,\F_q)$, of which there are $q^6$ corresponding to the
canonical matrix shape $\left(\begin{smallmatrix}
    1&0&*&*&*\\
    0&1&*&*&*
\end{smallmatrix}\right)$); there are
$q^2\cdot\gauss{3}{1}{q}\gauss{4}{2}{q}=q^2(q^2+q+1)(q^4+q^3+2q^2+q+1)$
planes $E$ meeting $S$ in a line (considering the $4$-flat $T=E+S$
and the line $L=E\cap S$ as fixed, these correspond to points
outside the line $S/L$ in $T/L\cong\PG(2,\F_q)$);\footnote{Here we
  have used $\gauss{4}{2}{q}=q^4+q^3+2q^2+q+1$, which follows from
  counting the partitions into at most $2$ parts of size $\leq 2$
  according to their sum: $0=0$, $1=1$, $2=2=1+1$, $3=2+1$, $4=2+2$.}
and finally, there are $\gauss{4}{3}{q}=\gauss{4}{1}{q}
=q^3+q^2+q+1$ planes contained in $S$.

Now let $\mathcal{C}$ be a set of planes in $\PG(6,\F_q)$ mutually
intersecting in at most a point (a plane subspace code in the
terminology of Section~\ref{sec:intro}). Fixing any solid $S$ in
$\PG(6,\F_q)$, we can count how many planes in $\mathcal{C}$ intersect
$S$ in a subspace of dimension $i\in\{0,1,2,3\}$. This leads to the
concept of ``spectra'' (or ``intersection vectors'') with respect to
solids, which already capture a great deal of structural information about
$\mathcal{C}$.

\begin{definition}
  \label{dfn:alpha}
  The \emph{spectrum} (or \emph{intersection vector}) of $\mathcal{C}$
  with respect to $S$ is defined as the $4$-tuple
  $\alpha(S)=\bigl(\alpha_0(S),\alpha_1(S),\alpha_2(S),\alpha_3(S)\bigr)$,
  $\alpha_i(S)=\#\{E\in\mathcal{C};\dim(E\cap S)=i\}$, of
  non-negative integers.
\end{definition}
The example counting problem discussed above amounts to determining the
spectrum of the whole plane set of $\PG(6,\F_q)$ with
respect to any solid, which turned out to be a constant independent of
$S$.\footnote{The latter also follows from the observation that
  $\GL(7,\F_q)$ acts transitively on the set of all plane-solid pairs
  $(E,S)$ with fixed intersection dimension $i$.} 

% We are primarily interested in sets of planes of $\PG(6,\F_q)$ mutually
% intersecting in at most a point. For simplicity we refer to these 
% simply as \emph{subspace codes} in this section.

\begin{lemma}
  \label{lma:alpha}
  Let $\mathcal{C}$ be a plane subspace code of size $M$ in $\PG(6,\F_q)$
  and $S$ any solid in $\PG(6,\F_q)$.
  The spectrum $\alpha=\alpha(S)$ of $\mathcal{C}$ with respect to $S$
  satisfies $\alpha_0+\alpha_1+\alpha_2+\alpha_3=M$ and the following system of
  linear inequalities:
  \begin{equation*}
    \label{eq:alpha}
    \begin{array}{rcrcrcrcl}
      % \alpha_0&+&\alpha_1&+&\alpha_2&+&\alpha_3&=&M\\
      \gauss{3}{1}{q}\cdot\alpha_0&+&q^2\alpha_1&&&&
      &\leq&q^8\cdot\gauss{3}{1}{q}\\
      &&(q+1)\alpha_1&+&(q^2+q)\alpha_2&&
      &\leq&q^3\cdot\gauss{3}{1}{q}\cdot\gauss{4}{1}{q}\\
      &&&&\alpha_2&+&\gauss{3}{1}{q}\cdot\alpha_3&\leq&\gauss{4}{2}{q}\\
      &&&&&&\alpha_3&\leq&1
    \end{array}
  \end{equation*}
\end{lemma}
The explicit form of all four inequalities is obtained
by inserting $\gauss{3}{1}{q}=q^2+q+1$, $\gauss{4}{1}{q}=q^3+q^2+q+1$ and
$\gauss{4}{2}{q}=q^4+q^3+2q^2+q+1=(q^2+1)(q^2+q+1)$.
\begin{proof}
  The equation $\alpha_0+\alpha_1+\alpha_2+\alpha_3=M$ is clear
  from the definition of the spectrum. The first three inequalities are proved
  by counting the line-plane pairs $(L,E)$ with $E\in\mathcal{C}$,
  $L\subset E$ and $\dim(L\cap S)=i$ for $i=0,1,2$, respectively, in
  two ways and using the fact that every line is contained in at most
  one plane of $\mathcal{C}$ (and hence counted at most once on the
  left-hand side). The right-hand side of the corresponding inequality
  gives the total number of lines $L$ with $\dim(L\cap S)=i$.
  Finally, since two distinct planes of $\mathcal{C}$ generate an at
  least $5$-dimensional space, $S$ can contain at most one plane of
  $\mathcal{C}$ and thus $\alpha_3\in\{0,1\}$.
\end{proof}
Lemma~\ref{lma:alpha} can be used to derive quite restrictive
conditions on the parameters of a putative $q$-analogue of the Fano
plane. This is the subject of Lemma~\ref{lma:q-analog}. For the
statement of the lemma recall that the cyclotomic polynomials
$\cpol_n(X)\in\Z[X]$, defined recursively by $X^n-1=\prod_{d\mid
  n}\cpol_d(X)$ for $n\in\N$, satisfy
$\cpol_p(X)=X^{p-1}+X^{p-2}+\dots+X+1$ for prime numbers $p$, as well
as $\cpol_6(X)=X^2-X+1$. In terms of cyclotomic polynomials the number
of points of $\PG(n-1,\F_q)$ is
$\gauss{n}{1}{q}=\frac{q^n-1}{q-1}=\prod_{d\mid n,d\neq 1}\cpol_d(q)$.
\begin{lemma}
  \label{lma:q-analog}
  If a $q$-analogue $\mathcal{F}_q$ of the Fano plane exists, it must
  have the following properties:
  \begin{enumerate}[(i)]
  \item The number of planes in $\mathcal{F}_q$ is
    \begin{align*}
      \#\mathcal{F}_q
      &=\cpol_7(q)\cpol_6(q)
      =(q^6+q^5+q^4+q^3+q^2+q+1)(q^2-q+1)\\&=q^8+q^6+q^5+q^4+q^3+q^2+1,
      \end{align*}
      with $\cpol_{6}(q)\cpol_3(q)
      =q^4+q^2+1$ planes passing through each point of $\PG(6,\F_q)$.
  \item The spectrum of $\mathcal{F}_q$ with respect to solids
    takes the two values
    \begin{align*}
      \alpha_0&=(q^8-q^7+q^3,q^7+q^6+q^5-q^3-q^2-q,q^4+q^3+2q^2+q+1,0),\\
      \alpha_1&=(q^8-q^7,q^7+q^6+q^5,q^4+q^3+q^2,1)
    \end{align*}
    with corresponding frequencies
    \begin{align*}
      f_0&=q^{12}+q^{10}+q^9+q^8+q^7+q^6+q^4,\\
      f_1&=q^{11} + q^{10} + 2q^{9} + 3q^{8} + 3q^{7} + 4q^{6} 
      + 4q^{5} + 3q^{4} + 3q^{3} + 2q^{2} + q + 1.
    \end{align*}
    % \footnote{As a consistency check note that the total number of
    %   solids (planes) in $\PG(6,\F_q)$ is $\gauss{7}{4}{q}=\gauss{7}{3}{q}
    %   =\cpol_7(q)\cpol_6(q)\cpol_5(q)$.}
  \end{enumerate}
\end{lemma}
\begin{proof}
  For a $q$-analogue of the Fano plane 
  the first three inequalities
  in Lemma~\ref{lma:alpha} are in fact equalities (for any solid $S$)
  and, conversely, this property (even if it holds only for one
  particular solid $S$) implies that $\mathcal{C}$ must be a
  $q$-analogue of the Fano plane.

  Further, the triangular shape of the system
  implies that each of the two possible choices $\alpha_3\in\{0,1\}$ leads to
  a unique solution for $\alpha_1,\alpha_2,\alpha_3$.

  In the first case ($\alpha_3=0$) we obtain
  \begin{align*}
    \alpha_2&=q^4+q^3+2q^2+q+1=\cpol_4(q)\cpol_3(q),\\
    \alpha_1&=\frac{1}{q+1}\left(q^3\gauss{3}{1}{q}\gauss{4}{1}{q}
      -q(q+1)\alpha_2\right)\\
    &=\frac{1}{q+1}\left(q^3\cdot\cpol_3(q)\cdot\cpol_4(q)\cpol_2(q)
      -q\cdot\cpol_2(q)\cdot\cpol_4(q)\cpol_3(q)\right)\\
    &=(q^3-q)\cpol_4(q)\cpol_3(q)
    =q\cdot\cpol_4(q)\cpol_3(q)\cpol_2(q)\cpol_1(q)\\
    &=q(q^4-1)(q^2+q+1)=q^7+q^6+q^5-q^3-q^2-q,\\
    \alpha_0&=q^8-\frac{q^2}{q^2+q+1}\cdot\alpha_1=q^8-q^7+q^3,
  \end{align*}
  as asserted. The second case ($\alpha_3=1$) is done similarly.
  
  Finally, a solid $S$ of $\PG(6,\F_q)$ has $\alpha_3(S)=1$ iff it
  contains a plane of $\mathcal{F}_q$. The number of such solids is
  \begin{align*}
    f_1&=\#\mathcal{F}_q\cdot\gauss{4}{1}{q}
    =\cpol_7(q)\cpol_6(q)\cpol_4(q)\cpol_2(q)\\
    &=q^{11} + q^{10} + 2q^{9} + 3q^{8} + 3q^{7} + 4q^{6} 
    + 4q^{5} + 3q^{4} + 3q^{3} + 2q^{2} + q + 1,
  \end{align*}
  and the number of solids with $\alpha_3(S)=0$ is
  \begin{align*}
    f_0&=\gauss{7}{4}{q}-f_1=\gauss{7}{3}{q}-f_1\\
    &=\cpol_7(q)\cpol_6(q)\cpol_5(q)-\cpol_7(q)\cpol_6(q)\cpol_4(q)\cpol_2(q)\\
    %&=q^4\cdot\cpol_7(q)\cpol_6(q)\\
    &=\#\mathcal{F}_q\cdot q^4=q^{12}+q^{10}+q^9+q^8+q^7+q^6+q^4,
  \end{align*}
  completing the proof.
\end{proof}
\begin{remark}
  \label{rmk:q-analog}
  More general results on the intersection structure of a putative
  $q$-analogue of the Fano plane can be found in
  \cite[Sect.~4]{kiermaier-pavcevic14}.
  
  Performing the same computations, mutatis mutandis, for putative Steiner
  systems $\steiner_q(2,3,v)$ with arbitrary ambient space dimension
  $v$ yields non-integral solutions and hence excludes the existence
  of an $\steiner_q(2,3,v)$ for $v\equiv 0,2,4,5\pmod{6}$. Thus an
  $\steiner_q(2,3,v)$ can exist only for
  $v\in\{7,9,13,15,19,21,25,27,\dots\}$. For the particular case
  $q=2$, $v=13$ existence 
  has been proved in \cite{braun-etal13}, providing the only known
  nontrivial example of a Steiner system over a finite field. 
  This remarkable result was the outcome of a computer search for
  Steiner systems $\steiner_2(2,3,13)$ invariant under the normalizer
  of a Singer subgroup of $\GL(13,\F_2)$, a group of order
  $(2^{13}-1)\cdot 13=106483$, and of course facilitated by the fact that
  Steiner systems $\steiner_2(2,3,13)$ with this additional structure
  exist.\footnote{An $\steiner_2(2,3,13)$ contains as many as
    $\frac{(2^{13}-1)(2^{12}-1)}{21}
    =1597245$ planes out of a total of $\gauss{13}{3}{2}=3269560515$
    planes in $\PG(12,\F_2)$, rendering any unrestricted search for such
    a structure completely infeasible.} 
\end{remark}

% \begin{proof}
%   (i) The usual double-counting arguments give
%   \begin{align*}
%     b&=\frac{\gauss{7}{2}{q}}{\gauss{3}{2}{q}}
%     =\frac{(q^7-1)(q^6-1)}{(q^2-1)(q-1)(q^2+q+1)}\\
%     &=(q^6+q^5+q^4+q^3+q^2+q+1)(q^2-q+1),\\
%     \intertext{where $q^4+q^2+1=(q^2+q+1)(q^2-q+1)$ was used,}
%     r&=\frac{\gauss{6}{1}{q}}{\gauss{2}{1}{q}}
%     =\frac{q^6-1}{(q-1)(q+1)}\\
%       &=q^4+q^2+1.
%   \end{align*}
%   (i) In Case~1 the blocks of $\mathcal{F}$ meeting $S$ in a line are
%   in 1-1 correspondence with the lines contained in $S$, whence 
%   $a_2=\gauss{4}{2}{q}=(q^2+1)(q^2+q+1)=q^4+q^3+2q^2+q+1$. The number
%   of blocks meeting $S$ in a fixed point $P$ is then
%   $q^4+q^2+1-(q^2+q+1)=q^4-q$, and
%   \begin{align*}
%     a_1&=(q^3+q^2+q+1)(q^4-q)=q^7+q^6+q^5-q^3-q^2-q,\\
%     a_0&=b-a_1-a_2=b-(q^7+q^6+q^5+q^4+q^2+1)=q^8-q^7+q^3.
%   \end{align*}
%   In Case~2 the block $W\in\mathcal{F}$ covers $q^2+q+1$ lines, all
%   contained in $S$, and
%   $a_2=(q^2+1)(q^2+q+1)-(q^2+q+1)=q^4+q^3+q^2$. The number of blocks
%   meeting $S$ in a point $P$ is then $q^4+q^2+1-(q^2+1)=q^4$ if $P\in
%   W$ ($q+1$ of the $q^2+q+1$ lines in $S$ through $P$ are covered by
%   the single block $W$), and as before $q^4-q$ if $P\in
%   S\setminus W$. It follows that
%   \begin{align*}
%     a_1&=(q^2+q+1)q^4+q^3(q^4-q)=q^7+q^6+q^5\\
%     a_0&=b-a_1-a_2-1=b-(q^7+q^6+q^5+q^4+q^3+q^2+1)=q^8-q^7
%   \end{align*}
%   in this case, and the proof of the lemma is complete.
% \end{proof}

\section{Augmented LMRD Codes}\label{sec:almrd}

The initial subspace code constructions by Koetter, Kschischang and
Silva \cite{koetter-kschischang08,silva-kschischang-koetter08} were
based on the observation that the dimension of the intersection of two
$k$-dimensional subspaces $U,V$ of $\F_q^v/\F_q$ with canonical
matrices of the special form $(\imat_k|\mat{A})$, $(\imat_k|\mat{B})$
can be expressed through the rank of the matrix
$\mat{A}-\mat{B}\in\F_q^{k\times(v-k)}$. In fact it is easily seen
that $U\cap
V=\bigl\{(\vek{x}|\vek{x}\mat{A});\vek{x}\in\kernel(\mat{A}-\mat{B})\bigr\}
\cong\kernel(\mat{A}-\mat{B})$ (the left kernel of $\mat{A}-\mat{B}$)
and thus $\dim(U\cap V)=k-\rank(\mat{A}-\mat{B})$.

From earlier work of Delsarte \cite{delsarte78a} (and independently
Gabidulin and Roth \cite{gabidulin85,roth91}) the maximum number of
matrices in $\F_q^{m\times n}$ having pairwise rank distance at least
$d$ is known to be $q^{(m-d+1)n}$, provided that $m\leq
n$.\footnote{The assumption $m\leq n$ imposes no essential
  restriction, since matrices can be transposed without changing the
  rank.}  Subsets $\mathcal{A}\subseteq\F_q^{m\times n}$ of size
$q^{(m-d+1)n}$ with $\rank(\mat{A}-\mat{B})\geq d$ for all pairs of
distinct $\mat{A},\mat{B}\in\mathcal{A}$ are known as $(m,n,m-d+1)$
\emph{maximum rank distance (MRD) codes}. Via the \emph{lifting
  construction}
$\mathcal{A}\to\mathcal{L}\subseteq\F_q^{m\times(m+n)}$,
$\mat{A}\mapsto\langle(\imat_m|\mat{A})\rangle$ they give rise to
subspace codes $\mathcal{L}$ in $\PG(m+n-1,\F_q)$ of size
$\#\mathcal{L}=\#\mathcal{A}=q^{(m-d+1)n}$, constant dimension $m$ and
maximum intersection dimension $m-d$, as we have indicated above. These
subspace codes are called \emph{lifted maximum rank distance (LMRD) codes}.

In the case of interest to us we can find $q^8$ matrices in
$\F_q^{3\times 4}$ at pairwise rank distance $\geq 2$ and lift these
to a plane LMRD code in $\PG(6,\F_q)$ of size $q^8$ with maximum
intersection dimension $1$. This gives the lower bound
$\#\mathcal{C}\geq q^8$ for the maximum size of a plane subspace code
in $\PG(6,\F_q)$, which is already of the same asymptotic
order as a putative $2$-analogue of the Fano plane
($\#\mathcal{F}_q=q^8+q^6+q^5+q^4+q^3+q^2+1$).

Following the work in
\cite{koetter-kschischang08,silva-kschischang-koetter08}, several
constructions have been proposed for augmenting LMRD codes without
increasing $t$. (Note that increasing $t$ sacrifices the
error-correction capabilities of the original subspace code.) All
these constructions are variants of the so-called
\emph{echelon-Ferrers} construction introduced in
\cite{etzion-silberstein09}, which combines subspace codes in
different Schubert cells of the corresponding Grassmannian in a
certain way.\footnote{``Schubert cell'' refers to the set of all
  subspaces whose canonical matrices have their pivot columns fixed.}
We will not delve into this further, but instead only mention that the
maximum size of an augmented LMRD code obtained in this way is
$\#\mathcal{C}=q^8+\gauss{4}{2}{q}=q^8+q^4+q^3+2q^2+q+1$ and provide a
different construction of such a code below.

In fact the bound $\#\mathcal{C}\leq q^8+\gauss{4}{2}{q}$ holds for
any plane subspace code in $\PG(6,\F_q)$ containing an LMRD code. This
is a consequence of the following lemma, which could be easily
generalized to arbitrary packet length $v$.
\begin{lemma}
  \label{lma:lmrd}
  Let $\mathcal{L}$ be a plane LMRD code in $\PG(6,\F_q)$ and
  $S=(0,0,0,*,*,*,*)$ the special solid defined by $x_1=x_2=x_3=0$.
  Then the planes in $\mathcal{L}$ cover all lines that are disjoint from
  $S$ (and no other lines). 
\end{lemma}
\begin{proof}
  A line $L$ disjoint from $S$ has a canonical matrix of the form
  $(\mat{Z}|\mat{B})$ with $\mat{Z}\in\F_q^{2\times 3}$ in canonical
  form and $\mat{B}\in\F_q^{2\times 4}$ arbitrary. Now let
  $\mathcal{A}$ be the matrix code corresponding to $\mathcal{L}$ and
  consider the map $\mathcal{A}\to\F_q^{2\times 4}$,
  $\mat{A}\mapsto\mat{ZA}$. Since $\rank(\mat{Z})=2$ and the minimum
  nonzero rank in $\mathcal{A}$ is $2$, this map must be injective,
  hence also surjective. Thus there exists $\mat{A}\in\mathcal{A}$
  such that $\mat{B}=\mat{ZA}$, implying
  $\cm(L)=\mat{Z}(\imat_3|\mat{A})$. The latter just says that $L$ is
  contained in the plane $\langle(\imat_3|\mat{A})\rangle\in\mathcal{L}$.
\end{proof}
With the aid of this lemma the bound $\#\mathcal{C}\leq
q^8+\gauss{4}{2}{q}$ is established as follows: A fortiori
$\mathcal{C}$ covers every line disjoint from $S$ and hence cannot
contain a plane meeting $S$ in a point (as such a plane would contain
lines disjoint from $S$). Thus, apart from the planes in
$\mathcal{L}$, it contains only planes meeting $S$ in a line or
planes entirely contained in $S$. The number of such planes is bounded
by the total number of lines in $S$, yielding the bound. (Moreover,
the bound can be achieved only if no plane of $\mathcal{C}$ is
contained in $S$.)

We close this section with an alternative construction of an augmented
plane LMRD code in $\PG(6,\F_q)$ of size
$q^8+\gauss{4}{2}{q}$. Such a code was first constructed in
\cite{trautmann-rosenthal10}. Our construction uses the existence of a
\emph{line packing} of $\PG(3,\F_q)$, which refers to a partition of
the line set into line spreads, where a \emph{line spread} is itself
defined as a partition of the point set into lines (the same as a
partial line spread that covers all points).\footnote{Line packings
  form a projective analogue of the standard resolution of the line
  set of an affine plane into parallel classes.}  Line packings of
$\PG(3,\F_q)$ exist for all prime powers $q>1$; cf.\
\cite{beutelspacher74,denniston11}. Since line spreads of
$\PG(3,\F_q)$ have size $q^2+1$ and
$\gauss{4}{2}{q}=(q^2+1)(q^2+q+1)$, the number of line spreads in a
line packing is $q^2+q+1$.

\begin{theorem}
  \label{thm:almrd}
  Any plane LMRD code $\mathcal{L}$ in $\PG(6,\F_q)$ can be
  augmented by $\gauss{4}{2}{q}=q^4+q^3+2q^2+q+1$ further planes to
  yield a plane subspace code $\mathcal{C}$ of
  size
  $\#\mathcal{C}=q^8+\gauss{4}{2}{q}=q^8+q^4+q^3+2q^2+q+1$.\footnote{We
    remind the reader one last time that all subspace codes considered
    (including LMRD codes) have $t=2$ (maximum intersection dimension $1$).}
\end{theorem}
\begin{proof}
  Choose a packing
  $\mathscr{P}=\{\mathcal{P}_1,\dots,\mathcal{P}_{q^2+q+1}\}$ of
  $\PG(S/\F_q)\cong\PG(3,\F_q)$, and let $\{P_1,\dots,P_{q^2+q+1}\}$
  be a set of points in $\PG(6,\F_q)$ forming a set of representatives
  for the $q^2+q+1$ $4$-flats containing $S$.\footnote{By this we mean
    $P_i\notin S$ and the $4$-flats $F_i=P_i+S$, $1\leq i\leq
    q^2+q+1$, account for all $4$-flats above $S$.} For $1\leq i\leq
  q^2+q+1$ connect the point $P_i$ to all $q^2+1$ lines $L_{ij}$ in
  $\mathcal{P}_i$ to form a set of $(q^2+1)(q^2+q+1)$ planes
  $E_{ij}=P_i+L_{ij}$. We claim that
  $\mathcal{C}=\mathcal{L}\cup\{E_{ij}\}$ has the required property.

  Clearly the ``new'' planes $E_{ij}$ cover no line disjoint from $S$
  and each line in $S$ exactly once. Now suppose, for contradiction,
  that $L$ is a line
  meeting $S$ in a point $P$ and contained in two different new planes
  $E=P_i+L_{ij}$, $E'=P_{i'}+L_{i'j'}$. Then $L$ must meet
  both $L_{ij}$ and $L_{i'j'}$ in $P$, whence $L_{ij}$ and $L_{i'j'}$
  intersect and $i\neq i'$. But the $4$-flats $F_i=L+S=F_{i'}$
  coincide, contradiction!
\end{proof}

The subspace code $\mathcal{C}$ of Theorem~\ref{thm:almrd} is quite
small in comparison with the codes constructed later in our main
theorems. But we feel that the construction method is of independent
interest and have included it for this reason.

\section{First Expurgating and Then Augmenting}\label{sec:aelmrd}

In this section we describe the basic idea used in \cite{smt:fq11proc}
to overcome the size restriction imposed on subspace codes containing
LMRD codes, tailored (and generalized) to the case of plane subspace codes in
$\PG(6,\F_q)$ with arbitrary $q$.

Given a plane LMRD code $\mathcal{L}$ in $\PG(6,\F_q)$, we
must obviously remove some of the $q^8$ planes in $\mathcal{L}$
first and then augment the resulting subcode
$\mathcal{L}_0\subset\mathcal{L}$ as far as possible. What is the best
way to do this? The ``removed'' set of planes
$\mathcal{L}_1$, of size $\#\mathcal{L}_1=M_1$ say, covers
$(q^2+q+1)M_1$ lines disjoint from the special solid
$S=(0,0,0,0,*,*,*)$, which become \emph{free lines} of $\mathcal{L}_0$
in the sense that any \emph{new plane} added to $\mathcal{L}_0$, which
contains only lines disjoint from $S$ that are free, will not
increase $t$ (i.e., introduce a multiple line cover). Of course we are
only interested in adding new planes which meet $S$ in a point at this
stage, since this is the only way to go beyond the construction
in Section~\ref{sec:almrd}. In this case,
provided an exact rearrangement of the free
lines into new planes is possible, the subspace code size will increase to
\begin{equation}
  \label{eq:aelmrd}
  q^8-M_1+\frac{(q^2+q+1)M_1}{q^2}=q^8+\frac{(q+1)M_1}{q^2},
\end{equation}
since new planes contain only $q^2$ lines disjoint from $S$. It is
clear that $M_1$ must be a multiple of $q^2$, and it has been shown in
\cite{smt:fq11proc} that $M_1=q^2$ is not feasible but $M_1=q^3$ can
be realized for a particular choice of $\mathcal{L}$ and as far as
only the rearrangement of lines disjoint from $S$ matters. (As an
additional requirement, the chosen new planes must not introduce a
multiple cover of a line meeting $S$ in a point.) We will now
develop the technical machinery needed to derive this result, adapted
to the case $v=7$.

Since the ambient space of $\PG(6,\F_q)$ does not matter (as long as
it is $7$-dimensional over $\F_q$), we take it as $V=W\times\F_{q^4}$,
where $W$ denotes the trace-zero subspace of $\F_{q^4}/\F_q$
(consisting of all $x\in\F_{q^4}$ satisfying
$\trace(x)=\trace_{\F_{q^4}/\F_q}(x)=x+x^q+x^{q^2}+x^{q^3}=0$). This
allows us to use the additional structure of $\PG(6,\F_q)$ imposed by
the extension field $\F_{q^4}$. In this model our special solid is
$S=\{0\}\times\F_{q^4}\cong\F_{q^4}$ (naturally); likewise, we make
the identification $W\times\{0\}\cong W$. Subspaces of $V/\F_q$ can be
parametrized in the form
\begin{equation}
  \label{eq:U(Z,t,f)}
  U=\bigl\{(x,f(x)+y);x\in Z,y\in T,f\in\Hom(Z,\F_{q^4}/T)\bigr\},  
\end{equation}
where
\begin{align*}
  Z&=\bigl\{x\in W;\exists y\in\F_{q^4}\text{ such that }(x,y)\in
  U\bigr\},\\
  T&=\bigl\{y\in\F_{q^4};(0,y)\in U\bigr\}\\
  %\mathcal{A}&=\bigl\{f\in\Hom(Z,\F_{q^4});\graph_f\subseteq U\bigr\}.
\end{align*}
and $f\colon Z\to\F_{q^4}$ is any $\F_q$-linear map whose graph (in
the sense of Real Analysis) $\graph_f=\bigl\{(x,f(x));x\in
Z\bigr\}$ is contained in $U$. The $\F_q$-subspaces
$Z\subseteq W$ (projection of $U$ onto $W$) and $T\subseteq\F_{q^4}$
(naturally isomorphic to the kernel $U\cap S$ of this projection) are uniquely
determined by $U$, while $f$ is only determined up to addition of an
$\F_q$-linear map with values in $T$ 
and may therefore be replaced by any element in the coset
$f+\Hom(Z,T)\in\Hom(Z,\F_{q^4})/\Hom(Z,T)\cong\Hom(Z,\F_{q^4}/T)$.\footnote{It
  goes without saying that ``$\Hom$'' denotes the set of $\F_q$-linear
  maps between the indicated $\F_q$-spaces, which forms an $\F_q$-space of its
  own with respect to the point-wise operations.}
We denote this parametrization by $U=U(Z,T,f)$, using sometimes the
subspaces $Z\times\{0\}$, $\{0\}\times T$ of $\PG(V/\F_q)$ in place of
$Z,T$, as indicated above.

Observe that the subspaces
disjoint from $S$ are precisely the graphs $\graph_f=U(Z,\{0\},f)$ of
$\F_q$-linear maps $f\colon Z\to\F_{q^4}$. At the other extreme, the
subspaces containing $S$ are of the form $U(Z,S,0)=Z\times S$.

The incidence relation between subspaces of $V/\F_q$
can also be described within this setting: $U(Z',T',f')\subseteq
U(Z,T,f)$ if and only if $Z'\subseteq Z$, $T'\subseteq T$ and 
%$f$ restricts to $f'$ on $Z'$ modulo $\Hom(Z',T')$
$f|_{Z'}-f'\in\Hom(Z',T)$.

% This coordinate-free description of $\PG(6,\F_q)$ may seem
% artificial in the first place, but it proves already useful for the
% description of a special class of MRD codes, so-called
% Gabidulin codes, which we will need in the sequel.

Now recall from Galois Theory that the powers
$\identity,\frob,\frob^2,\frob^3$ of the Frobenius automorphism
$\frob\colon\F_{q^4}\to\F_{q^4}$, $x\mapsto x^q$ of $\F_{q^4}/\F_q$
form a basis of $\End(\F_{q^4}/\F_q)$ over $\F_{q^4}$. This says that
every $\F_q$-linear map $f\colon\F_{q^4}\to\F_{q^4}$ is evaluation of
a unique linearized polynomial
$a(X)=a_0X+a_1X^q+a_2X^{q^2}+a_3X^{q^3}\in\F_{q^4}[X]$ of symbolic
degree $\leq 3$. For simplicity we write $x\mapsto f(x)$ as
$a_0x+a_1x^q+a_2x^{q^2}+a_3x^{q^3}$. The restriction map $f\mapsto
f|_W$ then gives that every element of $\Hom(W,\F_{q^4})$ is
represented uniquely as $a_0x+a_1x^q+a_2x^{q^2}$ for some
$a_0,a_1,a_2\in\F_{q^4}$ (since the linear maps vanishing on $W$ are
of the form $a(x+x^q+x^{q^2}+x^{q^3})$ with $a\in\F_{q^4}$). 

Next we
name various subspaces of $\Hom(W,\F_{q^4})$, which will subsequently
play an important role:
\begin{align*}
  \gabidulin&=\{a_0x+a_1x^q;a_0,a_1\in\F_{q^4}\},\\
  \rspace&=\{ax^q-a^qx;a\in\F_{q^4}\},\\
  \tspace&=\{ax^q-a^qx;a\in W\},\\
  \dspace(Z,P)&=r(ab^q-a^qb)^{-1}\langle ax^q-a^qx,bx^q-b^qx\rangle
\end{align*}
for a $2$-dimensional subspace $Z=\langle a,b\rangle$ of $W$ and a
point $P=\F_q(0,r)$ of the special solid $S$ (i.e.\
$r\in\F_{q^4}^\times$). The space $\gabidulin$ has minimum rank
distance $2$ (since $a_0x+a_1x^q\neq 0$ has at most $q$ zeros in $W$)
and size $\#\gabidulin=q^8$. It is therefore an MRD code.  We call it
the \emph{Gabidulin code}, since it is a basis-free version of a
member of the family of MRD codes constructed in \cite{gabidulin85},
which are nowadays commonly called Gabidulin codes. Further we have
$\dspace(Z,P)\subset\tspace\subset\rspace\subset\gabidulin$, $\tspace$
has constant rank $2$ (since $ax^q-a^qx$ has $1$-dimensional kernel
$\F_qa$ if $a\in W\setminus\{0\}$), $\rspace\setminus\tspace$ has constant
rank $3$, and $\dspace(Z,P)$ consists of all linear maps
$f\in\gabidulin$ satisfying $f(Z)\subseteq\F_qr$.\footnote{Of course
  $0\in\tspace$ has rank $0\neq 2$, but it is custom to refer to a
  matrix space as a constant-rank space if all nonzero matrices in the
  matrix space have the same rank.}

Finally we fix $\mathcal{L}=\{\graph_f;f\in\gabidulin\}$ for the
remainder of this article and call $\mathcal{L}$ the \emph{lifted
  Gabidulin code}. The reader may check that $f\mapsto\graph_f$ provides
a basis-free description of the lifting construction
(passing from matrix codes to subspace codes) and hence $\mathcal{L}$
is a plane LMRD code as needed for the subsequent discussion. 
%Since $\gabidulin\to\mathcal{L}$, $f\mapsto\graph_f$ is one-to-one, we

\begin{lemma}
  \label{lma:remove}
  For a set of planes $\mathcal{L}_1\subseteq\mathcal{L}$ let
  $\gabidulin_1\subseteq\gabidulin$ be the corresponding set of linear
  maps in the Gabidulin code. In order that the free lines determined
  by $\mathcal{L}_1$ can be rearranged into new planes meeting $S$ in
  a point, it is necessary and sufficient that $\#\gabidulin_1=mq^2$
  is a multiple of $q^2$ and for each
  $2$-dimensional subspace $Z\subset W$ there exist (not necessarily
  distinct) points
  $P_1,\dots,P_m$ on $S$ and linear maps $f_1,\dots,f_m\in\gabidulin$ such that
  \begin{equation*}
      \gabidulin_1=\biguplus_{i=1}^m\bigl(f_i+\dspace(Z,P_i)\bigr).
  \end{equation*}
\end{lemma}
Note that the condition requires $\gabidulin_1$ to be a union of
cosets of spaces $\dspace(Z,P)$ simultaneously in $q^2+q+1$ different
ways, one for each $2$-dimensional subspace $Z\subset W$. The number
of new planes in the rearrangement must be $m(q^2+q+1)$, but the
rearrangement itself is perhaps not uniquely determined by
$\gabidulin_1$. Moreover, the lemma does not say anything about
whether the rearrangement introduces a multiple
cover of some line meeting $S$ in a point.
\begin{proof}[Proof of the lemma]
  Lines $L$ disjoint from $S$ as well as new planes $N$ meeting $S$ in a
  point are contained in a unique hyperplane $H$ above $S$ ($H=L+
  S$ resp.\ $H=N+S$). ``Old'' planes $E\in\mathcal{L}$ are
  transversal to these hyperplanes and the $H$-section $E\mapsto E\cap
  H$ identifies $\mathcal{L}$ with the set of $q^8$ lines in
  $\PG(H)$ disjoint from $S$ (since $\mathcal{L}$ is an LMRD
  code). In terms of the parametrization $H=H(Z,\F_{q^4},0)$,
  $E=\graph_f$, $L=\graph_g$ the corresponding $H$-section is just
  restriction $g=f|_Z$. Thus we can look at each hyperplane above $S$
  separately. 

  Let $H$ be such a hyperplane and $Z$ the corresponding
  $2$-dimensional subspace of $W$. Planes in $H$ meeting $S$ in the
  point $P=\F_q(0,r)$ have the form $N=N(Z,\F_qr,g)$ with
  $g\in\Hom(Z,\F_{q^4})$ and contain the $q^2$ lines $L=\graph_h$,
  $h\in g+\Hom(Z,\F_qr)$, disjoint from $S$. Denoting by $f\in\gabidulin$
  the unique linear map such that $f|_Z=g$, we have that
  $f+\dspace(Z,P)$ restricts to $g+\Hom(Z,\F_qr)$ on $Z$. Hence the
  $mq^2$ free lines in $H$ determined by the planes in $\mathcal{L}_1$ can be
  rearranged into new planes $N(Z,\F_qr,g)$ iff $\gabidulin_1$ is a
  disjoint union of cosets of the form $f+\dspace(Z,P)$ with $P\in S$,
  $f\in\gabidulin_1$.
  %This already proves the lemma.
\end{proof}
Now observe that our distinguished space $\tspace$ contains one space
$\dspace(Z,P)$ for each $Z=\langle a,b\rangle\subset W$, viz.\
$\dspace(Z,P)$ with $P=\F_q(0,ab^q-a^qb)$. Hence
$\gabidulin_1=\tspace$ satisfies the conditions of
Lemma~\ref{lma:remove} with $m=q$, $P_1=\dots=P_q=P=\F_q(0,ab^q-a^qb)$
and $f_1,\dots,f_q$ a system of coset representatives for
$\tspace/\dspace(Z,P)$. A fortiori the same is true for any coset of
$\tspace$ in $\gabidulin$, and even for any disjoint union of
``rotated'' cosets $\biguplus_{j=1}^r(f_j+r_j\tspace)$ with
$r_j\in\F_{q^4}^\times$ and $f_j\in\gabidulin$.\footnote{For the
  latter the points $P_i$ vary not only with $Z$ but also with $j$.}

The next theorem, which closes this section, shows that if we take 
$\gabidulin_1=\rspace$, the distinguished subspace of order $q^4$
defined along with $\tspace$, then the corresponding rearrangement
into new planes does not introduce a multiple line cover and hence
results in a plane subspace code with $t=2$.

\begin{theorem}
  \label{thm:rspace}
  Let $\mathcal{C}$ be the set of planes in
  $\PG(W\times\F_{q^4})\cong\PG(6,\F_q)$ obtained by removing all
  planes $E=\graph_f$, $f\in\rspace$, from $\mathcal{L}$ and adding all
  planes of the form $N=N(Z,P,g)$ with $Z=\langle a,b\rangle\subset W$
  $2$-dimensional, $P=\F_q(0,ab^q-a^qb)$ (so $P$ depends on $Z$) and
  %$g\in f|_Z+\Hom\bigl(Z,\F_q(ab^q-a^qb)\bigr)$ 
  $g=f|_Z$ for some
  $f\in\rspace$. Then $\mathcal{C}$ forms a subspace code (i.e.,
  $t=2$) of size $\#\mathcal{C}=q^8+q^3+q^2$. Moreover, $\mathcal{C}$
  can be augmented by $\gauss{4}{2}{q}$ further planes meeting $S$ in
  a line to a subspace code $\widehat{\mathcal{C}}$ of
  size $\#\widehat{\mathcal{C}}=q^8+q^4+2q^3+3q^2+q+1$.
\end{theorem}
\begin{proof}
  Since $M_1=\#\rspace=q^4$, the rearrangement increases the size of
  the subspace code by $(q+1)M_1/q^2=q^3+q^2$. Thus
  $\#\mathcal{C}=q^8+q^3+q^2$, and it remains to show that
  $\mathcal{C}$ still has $t=2$. 

  By Lemma~\ref{lma:remove} and the definition of $\mathcal{C}$, the
  new planes $N=N(Z,P,g)$ added to
  $\mathcal{L}_0=\mathcal{L}\setminus\mathcal{L}_1$ cover each free
  line exactly once. Hence it suffices to check that no line meeting
  $S$ in a point is covered more than once.

  To this end we first we show that the map $\langle
  a,b\rangle\mapsto\F_q(ab^q-a^qb)$ (i.e.\ $Z\mapsto P$) is
  one-to-one. This implies that new planes in different hyperplanes
  above $S$ do not meet on $S$ and hence cannot intersect in a
  line. Suppose, by contradiction, that different subspaces 
  $Z_1,Z_2$ of $W$ correspond to the
  same point $P$. Since $\dim(Z_1\cap Z_2)=1$, we can write
  $Z_1=\langle a,b_1\rangle$, $Z_2=\langle a,b_2\rangle$. The $\F_q$-linear map
  $ax^q-a^qx\in\Hom(W,\F_{q^4})$ has kernel $\F_qa$ and hence maps $Z_1,Z_2$ to
  different $1$-dimensional subspaces
  $\F_q(ab_1^q-a^qb_1)\neq\F_q(ab_2^q-a^qb_2)$; contradiction!

  Next let $N_1=N(Z,P,g_1)$, $N_2=N(Z,P,g_2)$, $g_i=f_i|_Z$, be
  different new planes meeting $S$ in the same point $P$ (and hence
  with the same $Z$). Write $Z=\langle a,b\rangle$ and
  $f_1(x)-f_2(x)=u_0x+u_1x^q$. The planes $N_1,N_2$ have a point outside $S$
  (and hence a line through $P$) in common iff there exists $x\in
  Z\setminus\{0\}$ such that
  $f_1(x)-f_2(x)\in\F_q(ab^q-a^qb)$. Setting $x=\lambda
  a+\mu b$, this is equivalent to a nontrivial solution
  $(\lambda,\mu,\nu)\in\F_q^3$ of the equation
  \begin{equation*}
    \lambda(u_0a+u_1a^q)+\mu(u_0b+u_1b^q)+\nu(ab^q-a^qb)=0.
  \end{equation*}
  Thus $f_1,f_2\in\gabidulin$ determine new planes $N_1,N_2$
  satisfying $N_1\cap N_2=\{P\}$ for those choices of $Z=\langle
  a,b\rangle\subset W$
  (equivalently, for those choices of the hyperplane $H=Z+S$) for
  which $u_0a+u_1a^q$, $u_0b+u_1b^q$, $ab^q-a^qb$ are linearly
  independent over $\F_q$.\footnote{Viewed projectively, this requires
    that $f(x)=u_0x+u_1x^q$ maps the line $Z=\langle a,b\rangle$ to
    another line $Z'=f(Z)$ of $\PG(\F_{q^4}/\F_q)$ and the 
    point $\F_q(ab^q-a^qb)$ corresponding to $Z$ is not on $Z'$.}

  With these preparations we can now prove that $\mathcal{C}$ still
  has $t=2$. For $f_1,f_2\in\mathcal{R}$ we have
  $f_1-f_2\in\mathcal{R}$ and hence of the form $ux^q-u^qx$. If
  $f_1,f_2$ are in different cosets of $\dspace(Z,P)$ then $u\notin
  Z$. The equation $\lambda(ua^q-u^qa)+\mu(ub^q-u^qb)+\nu(ab^q-a^qb)=0$
  can be rewritten as
  \begin{equation*}
    \begin{vmatrix}
      a&b&u\\
      a^q&b^q&u^q\\
      -\mu&\lambda&\nu
    \end{vmatrix}=0.
  \end{equation*}
  If $(\lambda,\mu,\nu)$ is nonzero then
  using the linear dependence of the rows of this matrix we can
  express the conjugates $(a^{q^i},b^{q^i},u^{q^i})$ as linear
  combinations (with coefficients in $\F_{q^4}$) of $(a,b,u)$ and
  $(-\mu,\lambda,\nu)\in\F_q^3$. This shows that the $4\times 3$
  matrix formed from the conjugates of $(a,b,u)$ has rank $2$ and
  implies that $a,b,u$ are linearly dependent over $\F_q$;
  contradiction. Thus $\mathcal{C}$ has the required property.

  % Then, since $f_1,f_2\in\rspace$ and
  % $f_1-f_2\notin\dspace(Z,P)=\{ux^q-u^qx;u\in Z\}$, we have
  % $f_i(x)=u_ix^q-u_i^qx$ for some $u_1,u_2\in\F_{q^4}$ with
  % $c:=u_1-u_2\notin Z$. If $N_1,N_2$ have a line, and hence a point
  % $\F_q(x,y)\notin S$, in common then there exists $x\in
  % Z\setminus\{0\}$ such that
  % $f_1(x)-f_2(x)=cx^q-c^qx\in\F_q(ab^q-a^qb)$, where as usual
  % $Z=\langle a,b\rangle$. Since $c\notin Z$, we have $cx^q-c^qx\neq 0$
  % and hence $\F_q(cx^q-c^qx)=\F_q(ab^q-a^qb)$. Writing $x=\lambda
  % a+\mu b$ with $\lambda,\mu\in\F_q$ we may assume $\mu\neq 0$
  % (otherwise swap $a$ and $b$) and then obtain
  % $b=\mu^{-1}x-\mu^{-1}\lambda a$, $ab^q-a^qb=\mu^{-1}(ax^q-a^qx)$ and
  % hence also $\F_q(cx^q-c^qx)=\F_q(ax^q-a^qx)$. On the other hand,
  % $a,x,c$ are linearly independent over $\F_q$ and hence $y\mapsto
  % yx^q-y^qx$ (which has kernel $\F_qx$) maps $\langle a,c\rangle$ onto
  % a $2$-dimensional subspace of $\F_{q^4}$. The latter is obviously
  % inconsistent with the former and, in all, implies that $N_1$, $N_2$
  % intersect only in $P$.\footnote{Alternatively, one can show that the linear
  %   independence of $a,b,c$ over $\F_q$ (which is true by assumption)
  %   implies the linear independence of $ab^q-a^qb$, $bc^q-b^qc$,
  %   $ca^q-c^qa$ and from this derive the desired contradiction; cf.\
  %   \cite[Ex.~4]{smt:fq11proc} for details.} 

  The augmented subspace code $\widehat{\mathcal{C}}$ is constructed
  in the same way as in the proof of Theorem~\ref{thm:almrd}. The only
  thing that needs to be checked is that each $4$-flat $F$ above $S$
  contains a point $Q\notin S$ that is not covered by any new plane
  $N\in\mathcal{C}$. Equivalently, for any $x\in W\setminus\{0\}$ the
  new planes $N=N(Z,\F_qr,g)$ with $x\in Z$ do not cover all $q^4$ points
  $\F_q(x,y)$, $y\in\F_{q^4}$. This property will now be verified
  through explicit computation.

  A $2$-dimensional subspace $Z\subset W$ containing $x$ has the form
  $Z=\langle a,x\rangle$ with $a\in W$ and $ax^q-a^qx\neq 0$. The
  points $\F_q(x,y)$ 
  covered by the $q^2$ new planes corresponding to $Z$ have
  $y=ux^q-u^qx+\mu(ax^q-a^qx)=(u+\mu a)x^q-(u+\mu a)^qx$ for
  $u\in\F_{q^4}/Z$, $\mu\in\F_q$. It follows that $y$ takes precisely
  the $q^3$ values in the image $I$ of the linear map $c\mapsto
  cx^q-c^qx$, which has kernel $\F_qx$. In other words, the points in
  the $4$-flat $F=(\F_qx)\times S$ covered by the new planes in
  $\mathcal{C}$ form the affine part of a solid, viz.\ $(F_qx)\times I$,
  with plane at infinity $\{0\}\times I$. In particular, there
  are $q^4-q^3$ valid choices for the point $Q$.
  This completes the proof of the Theorem~\ref{thm:rspace}.
  % Inserting $x=\lambda a+\mu b$
  % ($\lambda,\mu\in\F_q$) gives
  % \begin{equation*}
  %   \lambda(ca^q-c^qa)+\mu(cb^q-c^qb)\in\F_q(ab^q-a^qb).
  % \end{equation*}
  % The proof is finished by showing that the linear independence of
  % $a,b,c$ over $\F_q$ (which is true by assumption)
  % implies the linear independence of $ab^q-a^qb$, $bc^q-b^qc$,
  % $ca^q-c^qa$ and hence gives the desired contradiction. Indeed,
  % since $a,b,c$ are linearly independent, we have
  % \begin{equation*}
  %   \begin{vmatrix}
  %     a&b&c\\
  %     a^q&b^q&c^q\\
  %     a^{q^2}&b^{q^2}&c^{q^2}
  %   \end{vmatrix}\neq 0.
  % \end{equation*}
\end{proof}
In the binary case $q=2$ the size of the augmented subspace code in
Theorem~\ref{thm:rspace} is $\#\widehat{\mathcal{C}}=303$, falling
short by $1$ of the corresponding code in
\cite{kohnert-kurz08}. On the other hand, $\#\widehat{\mathcal{C}}$
strictly exceeds the bound imposed on codes containing an LMRD code for every
$q$, showing already the effectiveness of our approach. However,
this is not the end of the story; Theorem~\ref{thm:rspace} will be 
improved upon later.

\section{An Attempt to Construct a $q$-Analogue and its 
  Failure}\label{sec:attempt}

In this section we apply the method developed in the previous section
to the construction problem for $q$-analogues of the Fano plane. The
attempt eventually fails for every $q$ but produces the largest known
plane subspace codes in $\PG(6,\F_q)$.

We start with a few words on automorphisms of subspace codes in
$\PG(V/\F_q)$. The group $G=\GL(V/\F_q)$ obviously acts on plane
subspace codes in $\PG(V/\F_q)$, but is by way too large for our
purpose. The stabilizer $G_S$ of our special solid $S$ in $\GL(V/\F_q)$
consists of all maps $L$ of the form $(x,y)L=(xL_{11},xL_{12}+yL_{22})$
with $L_{11}\in\GL(W/\F_q)$, $L_{22}\in\GL(\F_{q^4}/\F_q)$ and
$L_{12}\in\Hom(W,\F_{q^4})$. The map $L$ sends a plane $E=\graph_f$
disjoint from $S$ to $\graph_g$ with
$g=L_{11}^{-1}fL_{22}+L_{11}^{-1}L_{12}$ (composition of maps is from
left to right for the moment), so that on the corresponding maps
$f\in\Hom(W,\F_{q^4})$ it affords the group of all ``affine''
transformations $f\mapsto A\circ f\circ B+C$ with $A\in\GL(\F_{q^4}/\F_q)$,
$B\in\GL(W/\F_q)$ and $C\in\Hom(W,\F_{q^4})$.

The group $G_S$ is still too large for our purpose, but we have that the
Gabidulin code $\gabidulin$ is invariant under the subgroup consisting
of all maps $f\mapsto rf$ with $r\in\F_{q^4}^\times$, which acts as a
Singer group on the projective space
$\PG(S/\F_q)\cong\PG(3,\F_q)$. This group, or rather the corresponding
subgroup $\Sigma\leq\GL(V/\F_q)$ consisting of all maps
$(x,y)\mapsto(x,ry)$ with $r\in\F_{q^4}^\times$, is suitable for our
purpose.\footnote{Viewed as collineation group, $\Sigma$ has order
  $q^4-1$ (not the same as the Singer group).}
It is our next goal to make the expurgation-augmentation process
of Section~\ref{sec:aelmrd} invariant under $\Sigma$.

How large should the set $\mathcal{L}_1$ of removed planes be for a
putative $q$-analogue $\mathcal{F}_q$\,? We can arrange coordinates in
such a way that $S$ does not contain a block of $\mathcal{F}_q$ and
hence $q^8-q^7+q^3$ blocks are disjoint from $S$; cf.\
Lemma~\ref{lma:q-analog}. This requires
\begin{equation*}
\#\mathcal{L}_1=q^7-q^3=q^3(q^4-1)=(q^4-q^3)(q^3+q^2+q+1)
\end{equation*}
and the number of new planes through each point $P\in S$ to be
$(q^4-q^3)(q^2+q+1)/q^2=q^4-q$.\footnote{As a consistency check, use
  that this number can also be obtained by subtracting from the total
  number $q^4+q^2+1$ of blocks through $P$ (cf.\
  Lemma~\ref{lma:q-analog}) the number $q^2+q+1$ of blocks that meet
  $S$ in a line through $P$. Indeed, $q^4+q^2+1-(q^2+q+1)=q^4-q$.}
Hence a $\Sigma$-invariant construction of $\mathcal{F}_q$ is at least
conceivable and, even better, there is a canonical candidate for a
$\Sigma$-invariant subset $\gabidulin_1\subset\gabidulin$ of the
appropriate size, viz.\ the union of all ``rotated'' cosets
$r(f+\tspace)$ with $f\in\rspace\setminus\tspace$ and
$r\in\F_{q^4}^\times$.\footnote{The spaces $r\tspace$ itself cannot be
  used, since these are not disjoint.}  A moment's reflection shows
that this set $\gabidulin_1$ consists precisely of all binomials
$a_0x+a_1x^q$ with $1$-dimensional kernel in $\F_{q^4}/\F_q$
complementary to $W$ (thus the rank in $\Hom(W,\F_{q^4})$ is $3$).
The complementary subset
$\gabidulin_0=\gabidulin\setminus\gabidulin_1$ consists of $0$, the
$2(q^4-1)$ monomials $rx$, $rx^q$ with $r\in\F_{q^4}^\times$, the
$(q^4-1)(q^2+q+1)$ binomials 
$r(ux^q-u^qx)$ with $r\in\F_{q^4}^\times$ and $u\in W\setminus\{0\}$
(these have rank $2$ in $\Hom(W,\F_{q^4})$) and
$(q^4-1)(q^3+q^2+q+1)(q-2)$ binomials $a_0x+a_1x^q$ with no nontrivial
zero in $\F_{q^4}$. The set $\gabidulin_1$ decomposes as
\begin{equation*}
  \gabidulin_1=\biguplus_{r\in\F_{q^4}^\times/\F_q^\times}r(\rspace\setminus\tspace), 
\end{equation*}
showing that the $(q-1)\times\frac{q^4-1}{q-1}=q^4-1$ cosets 
$r(f+\tspace)$ used are pairwise disjoint, as needed for the
construction.

New planes are defined by connecting the free lines $L$ in the planes
corresponding to $f+\tspace$ to the points $P=\F_q(0,ab^q-a^qb)$,
where $Z=\langle a,b\rangle\subset W$ is the $2$-dimensional 
subspace determined by
the hyperplane $H=L+S=Z\times S$ (the same definition as in
Section~\ref{sec:aelmrd}), and rotating: Free lines in the planes
corresponding to $r(f+\tspace)$, $r\in\F_{q^4}^\times$, are connected to
$rP=\F_q\bigl(0,r(ab^q-a^qb)\bigr)$ in the same way. The collection
$\mathcal{N}$ of $(q^4-q)(q^3+q^2+q+1)$ new planes determined in this
way is certainly $\Sigma$-invariant and contains $q^4-q$ planes
meeting $S$ in any particular point $P$. By construction,
$\mathcal{N}$ forms an exact cover of the free lines determined by
$\mathcal{L}_1$ (and $\mathcal{L}_0\cup\mathcal{N}$ forms an exact
cover of all lines disjoint from $S$), but $\mathcal{N}$ may cover
some lines meeting $S$ in a point more than once.

If for some value of $q$ 
the set $\mathcal{L}_0\cup\mathcal{N}$ still had $t=2$,
then the present construction would have been a big step towards the desired
$q$-analogue $\mathcal{F}_q$, leaving only the task to augment it by
$\gauss{4}{2}{q}$ further planes meeting $S$ in a line. Unfortunately,
however, it turns out that $\mathcal{L}_0\cup\mathcal{N}$ never has
$t=2$, rendering a construction of a $q$-analogue $\mathcal{F}_q$ in
this way impossible. This negative result will follow
from our subsequent analysis, which on the other hand will tell us
precisely 
how many planes should be removed from $\mathcal{L}_0\cup\mathcal{N}$ in
order to restore $t=2$. Fortunately, this number turns out
to be rather small. 

Let $\mathcal{N}_1\subset\mathcal{N}$ be the set of
$q^4-q=(q-1)(q^3+q^2+q)$ new planes passing through the special point
$P_1=\F_q(0,1)$. We are interested in finding the largest subset(s)
$\mathcal{N}_1'\subseteq\mathcal{N}_1$ consisting of planes mutually
intersecting in $P_1$. Denoting by $M_1'$ the maximum size of such a
subset $\mathcal{N}_1'$, it is clear from the preceding development
and $\Sigma$-invariance of $\mathcal{N}$ that $\mathcal{L}_0$ can then
be augmented by a subset $\mathcal{N}'$ of size
$M'=M_1'(q^3+q^2+q+1)$ without increasing $t$. If $\mathcal{N}'_1$ is
invariant under the subgroup of $\Sigma$ corresponding to
$\F_q^\times$ then
$\mathcal{N}'$ may be taken in the form
$\mathcal{N}'=\biguplus_{L\in\Sigma}L(\mathcal{N}_1')
=\biguplus_{r\in\F_{q^4}^\times/\F_q^\times}r\mathcal{N}_1'$,
making the augmented subspace code
$\mathcal{C}=\mathcal{L}_0\cup\mathcal{N}'$ again
$\Sigma$-invariant.\footnote{``$r\mathcal{N}_1'$ refers to the image
  of $\mathcal{N}_1'$ under $(x,y)\mapsto(x,ry)$.} 
  %Note that $r\mathcal{N}_1'=\mathcal{N}_1'$ for $r\in\F_q^\times$.}
If $\mathcal{N}_1'$ is not uniquely determined then there are many
further choices for $\mathcal{N}'$, which could lead to better overall
subspace codes during the final augmentation step.\footnote{Later we
  will see that the number of choices for $\mathcal{N}_1'$ is
  at least $(q^2)^{q^3+q^2+q+1}$;\ cf. Section~\ref{sec:ext}.}
  %where $t$ denotes the number of choices for $\mathcal{N}_1'$.}

Before writing down $\mathcal{N}_1$ in explicit form we will introduce some
further terminology. Relative to a $2$-dimensional subspace $Z\subset
W$, the letters $a,b,c,d$ will henceforth denote a basis of
$\F_{q^4}/\F_q$ such that $Z=\langle a,b\rangle$, $W=\langle
a,b,c\rangle$ and $\trace(d)=1$.\footnote{The element $d$ can be fixed
  once and for all, but $c$ depends on $Z$, of course.} Further we
set $\dickson(x,y)=xy^q-x^qy=\left|
  \begin{smallmatrix}
    x&y\\x^q&y^q
  \end{smallmatrix}\right|$ for $x,y\in\F_{q^4}$, which constitutes an
$\F_q$-bilinear, antisymmetric map with right annihilators
$\bigl\{y\in\F_{q^4};\dickson(x,y)=0\bigr\}=\F_qx$ and corresponding
right images
$\dickson(x,\F_{q^4})=\bigl\{z\in\F_{q^4};\trace(x^{-q-1}z)=0\bigr\}=x^{q+1}W$
(provided that $x\neq 0$). The latter follows from Hilbert's Satz~90,
using $z=xy^q-x^qy\iff x^{-q-1}z=(y/x)^q-y/x$. Since
$\dickson(x,y)=x\prod_{\lambda\in\F_q}(y-\lambda x)$, we also have
that $\F_q\dickson(x,y)$ depends only on the line $L=\langle
x,y\rangle$ of $\PG(\F_{q^4}/\F_q)$ (provided that $\F_qx\neq\F_qy$)
and is computed as the product of all points of $L$ in
$\F_{q^4}^\times/\F_q^\times$. Accordingly, we can write $\dickson(L)$
for $\F_q\dickson(x,y)$ and thus have a well-defined map
$L\mapsto\dickson(L)$ from lines to points of $\PG(\F_{q^4}/\F_q)$. As
shown above, $L\mapsto\dickson(L)$ maps the line pencil through
$\F_qx$ bijectively onto the plane $x^{q+1}W$,\footnote{This fact was
  already used implicitly in some proofs.} but we also have the
following

\begin{lemma}
  \label{lma:dickson}
  $L\mapsto\dickson(L)$ maps the lines contained in any plane $E$ of
  $\PG(\F_{q^4}/\F_q)$ bijectively onto the points of another plane
  $E'$. If $\uone\in\F_{q^4}^\times$ satisfies $\uone^q=-\uone$ then
  $(aW)'=a^{q+1}\uone W$ for $a\in\F_{q^4}^\times$.
\end{lemma}
Note that $\uone^q=-\uone$, or $\uone^{q-1}=-1$, is equivalent to
$\uone\in\F_q^\times$ for even $q$ and to
$\uone\notin\F_q^\times\wedge\uone^2\in\F_q^\times$ for odd $q$. In the
latter case $\F_q^\times\uone$ is the unique element of order $2$ in
$\F_{q^4}^\times/\F_q^\times$. Further note that every plane of
$\PG(\F_{q^4}/\F_q)$ has the form $aW$ for some $a\in\F_{q^4}^\times$
(by Singer's Theorem).
\begin{proof}
  Since any two lines in $E$ intersect and $L\mapsto\dickson(L)$ is
  injective on line pencils, it is clear that the $q^2+q+1$ points
  $\dickson(L)$ for $L\subset E$ are distinct.

  Now consider the special plane
  $E=W=\bigl\{x\in\F_{q^4};x+x^q+x^{q^2}+x^{q^3}=0\bigr\}$. For
  $x,y\in W$ we have
  \begin{align*}
    \trace\bigl(\uone\dickson(x,y)\bigr)&=\trace(\uone xy^q-\uone x^qy)\\
    &=\uone xy^q-\uone x^qy^{q^2}+\uone x^{q^2}y^{q^3}-\uone x^{q^3}y
    -(\uone x^qy-\uone x^{q^2}y^q+\uone x^{q^3}y^{q^2}-\uone xy^{q^3})\\
    &=\uone(x+x^{q^2})(y^q+y^{q^3})-\uone(x^q+x^{q^3})(y+y^{q^2})\\
    &=\uone(x+x^{q^2})(y^q+y^{q^3})+\uone(x+x^{q^2})(y+y^{q^2})\\
    &=\uone(x+x^{q^2})(y+y^q+y^{q^2}+y^{q^3})=0
  \end{align*}
  and hence $\dickson(x,y)\in \uone^{-1}W=\uone W$. Thus $W'=\uone W$, and then
  $\dickson(ax,ay)=a^{q+1}\dickson(x,y)$ yields $(aW)'=a^{q+1}\uone W$. 
\end{proof}
Finally, for a plane $E$ in $\PG(\F_{q^4}/\F_q)$ we define $\dickson(E)$
as the product of all points on $E$ in $\F_{q^4}^\times/\F_q^\times$
(this yields a map $E\mapsto\dickson(E)$ from planes to points of
$\PG(\F_{q^4}/\F_q)$ and is completely analogous to
the case of lines), and in the case $E\neq W$ another
projective invariant $\sickson(E)$ as
\begin{equation}
  \label{eq:sickson}
  \sickson(E)=\frac{\dickson(E)}{\dickson(Z)^{q+1}},\quad\text{where
    $Z=E\cap W$}.
\end{equation}
The reason for this extra definition will become clear in a moment (cf.\
the subsequent Lemma~\ref{lma:collision}).

Now we turn to the description of the new planes in $\mathcal{N}_1$.
By the reasoning in Section~\ref{sec:aelmrd} and since
$\dspace(Z,P_1)=\dspace(\langle
a,b\rangle,P_1)=\dickson(a,b)^{-1}\langle ax^q-a^qx,bx^q-b^qx\rangle$,
the planes in $\mathcal{N}_1$ are parametrized as $N=N(Z,P_1,g)$, where
$Z\subset W$ is $2$-dimensional and 
$g\colon Z\to\F_{q^4}$ is of the form
\begin{equation*}
  g(x)=\dickson(a,b)^{-1}\left(\lambda(dx^q-d^qx)+\mu(cx^q-c^qx)\right)
  =\frac{\dickson(\lambda d+\mu c,x)}{\dickson(a,b)}
\end{equation*}
with $\lambda\in\F_q^\times$,
$\mu\in\F_q$ ($q^4-q$ choices for $N$), and cover the $(q+1)q$ points
\begin{equation}
  \label{eq:cover}
  \F_q\left(x,\frac{\dickson(\lambda d+\mu
      c,x)}{\dickson(a,b)}+\nu\right),\quad\F_qx\in Z,\;\nu\in\F_q
\end{equation}
outside $S$.

We call a pair of new planes $N,N'\in\mathcal{N}_1$ a \emph{collision}
if $N,N'$ have a point outside $S$ (and hence a line through $P_1$) in
common. Such collisions are precisely the obstructions to adding $N,N'$
simultaneously to the expurgated LMRD code
$\mathcal{L}_0=\mathcal{L}\setminus\mathcal{L}_1$ of size
$q^8-q^7+q^3$. From Theorem~\ref{thm:rspace} (and its ``rotated''
analogues, so-to-speak) we know that collisions between $N=N(Z,P_1,g)$
and $N'=N(Z',P_1,g')$ can occur only if $Z\neq Z'$. In this case
$Z\cap Z'=\F_qz$ is a single point, so that every collision takes the form
\begin{equation}
  \label{eq:collision}
  \frac{\dickson(\lambda d+\mu
      c,z)}{\dickson(a,z)}+\nu=\frac{\dickson(\lambda' d+\mu'
      c,z)}{\dickson(a',z)}+\nu'
\end{equation}
with $z,a,a'$ spanning $W$. Rewriting the denominator as
$\dickson(a,z)$ makes the actual correspondence
$(Z,\lambda,\mu)\mapsto N$ depend on $z$. However, since
$\dickson(a,b)$ and $\dickson(a,z)$ differ only by a factor in
$\F_q^\times$, this dependence disappears in the projective view, where
$Z$ and the point $\F_q(\lambda d+\mu c)$ correspond collectively to a
set of $q-1$ new planes, viz.\ $N(Z,P_1,\F_q^\times g)$ with
$g(x)=\dickson(\lambda d+\mu c,x)/\dickson(a,b)$.\footnote{Of course
  this remark also applies when changing the generators $a,b$ of $Z$.}

Further note that setting $E=Z+\F_q(\lambda d+\mu c)$ 
% (the plane generated by $Z$ and $\F_q(\lambda d+\mu c)$)
gives a parametrization
of the $q^4-q=(q-1)(q^3+q^2+q)$ new planes in $\mathcal{N}_1$, $q-1$
planes at a time, by the $q^3+q^2+q$ planes $E\neq W$ of
$\PG(\F_{q^4}/\F_q)$.\footnote{Since the line $\langle c,d\rangle$ is skew
  to $Z$, the $q$ points $\F_q(d+\mu c)$, $\mu\in\F_q$, determine the
  $q$ planes $E\neq W$ above $Z$. Replacing $\lambda d+\mu c$ by
  $\lambda d+\mu c+\alpha a+\beta b$ has no effect on the plane
  $N(Z,P_1,g)$, since
  $\dickson(a,x),\dickson(b,x)\in\F_q\dickson(Z)$ for $x\in Z$ and
  hence $g$ is only changed inside the coset $g+\Hom(Z,\F_q)$.}

% Further note that $\F_q\dickson(a,x)=\dickson(Z)$ and
% $\F_q\dickson(\lambda d+\mu c,x)=\dickson(L)$, where $L=\langle\lambda
% d+\mu c,x\rangle$ is the line in $\PG(\F_{q^4}/\F_q)$ connecting
% $\F_q(\lambda d+\mu c)$ to $\F_qx$. Hence
% $\dickson(L)/\dickson(Z)\in\F_{q^4}^\times$ is also a well-defined
% point of $\PG(\F_{q^4}/\F_q)$. Since $Z$ and $L$ intersect in $\F_qx$,
% we have $\dickson(L)/\dickson(Z)\neq\F_q$ (as explained above).

\begin{lemma}
  \label{lma:collision}
  % The $2$-dimensional space $\dickson(L)/\dickson(Z)+\F_q$ (the line
  % conecting the points $\dickson(L)/\dickson(Z)$ and $\F_q=\F_q1$ in
  % $\PG(\F_{q^4}/\F_q)$) depends only on the plane $E=\langle
  % L,Z\rangle$ (a plane intersecting $W$ in a line) and the
  % intersection point $\F_qx=L\cap Z$. Writing
  % $\sickson(E,\F_qx)=\dickson(L)/\dickson(Z)+\F_q$, we have the folloing two
  % cases for collisions between $N$ and $N'$ to consider:
  Let $N=N(Z,P_1,g)$, $N'=N(Z',P_1,g')$ be planes in $\mathcal{N}_1$
  parametrized by distinct planes $E,E'$ of $\PG(\F_{q^4}/\F_q)$ in the fashion
  just described. Collisions between
  any of the $2(q-1)$ planes in $N(Z,P_1,\F_q^\times g)\uplus
  N(Z',P_1,\F_q^\times g')$ fall into the following two cases:
  \begin{enumerate}[(i)]
  \item $\sickson(E)\neq\sickson(E')$. In this case there are no
    collisions among the planes in
    $N(Z,P_1,\F_q^\times g)\uplus N(Z',P_1,\F_q^\times g')$.
  \item $\sickson(E)=\sickson(E')$. In this case any new plane in
    $N(Z,P_1,\F_q^\times g)$ collides with a unique new plane
    in $N(Z',P_1,\F_q^\times g')$ and vice versa, and we can select a
    maximum of $q-1$ mutually non-colliding planes from
    $N(Z,P_1,\F_q^\times g)\uplus N(Z',P_1,\F_q^\times g')$.
  \end{enumerate}
\end{lemma}
\begin{proof}
  First suppose $Z=Z'$. In this case there are no
  collisions, and we must show $\sickson(E)\neq\sickson(E')$ or,
  equivalently, $\dickson(E)\neq\dickson(E')$. The planes of
  $\PG(\F_{q^4}/\F_q)$ have the form $rW$ with $r$ running through a
  system of coset representatives for $\F_q^\times$ in
  $\F_{q^4}^\times$, and clearly
  $\dickson(rW)=r^{q^2+q+1}\dickson(W)$. Since
  $\gcd(q^3+q^2+q+1,q^2+q+1)=1$, $E\mapsto\dickson(E)$ is a bijection
  and the result follows.

  Now suppose $Z\neq Z'$ and set $Z\cap Z'=\F_qz$. Assuming w.l.o.g.\
  $g(x)=\dickson(d+\mu c,x)/\dickson(a,x)$, we have from
  \eqref{eq:cover} that the points on $N(Z,P_1,\lambda g)$ of the form
  $\F_q(z,y)$ are those with $y\in\lambda g(z)+\F_q$,
  i.e.\ the $q$ points $\neq P_1$ on the line through $\F_q\bigl(z,\lambda
  g(z)\bigr)$ and $P_1$. Hence the points $\F_q(z,y)$ on the planes in
  $N(Z,P_1,\F_q^\times g)$ are those with $y\in\F_q^\times g(z)+\F_q$, an
  orbit of the affine group $\AGL(1,\F_q)=\{u\mapsto \lambda
  u+\nu;\lambda\in\F_q^\times,\nu\in\F_q\}$ acting on $\F_{q^4}$. The orbits
  corresponding to $N,N'$ are either disjoint and there are no collisions,
  or the orbits coincide and the planes in $N(Z,P_1,\F_q^\times g)$
  and $N(Z',P_1,\F_q^\times g')$ are matched up in pairs covering the
  same line $L$ through $P_1$ and a point of the form $\F_q\bigl(z,\lambda
  g(z)\bigr)$. In this case we can select at most one plane from
  each matching pair without introducing collisions. If we do so, the
  selected planes will cover the same lines $L$ as the corresponding
  planes in $N(Z,P_1,\F_q^\times g)$, say,
  and hence there is no obstruction to selecting exactly one plane from
  each pair.

  It remains to show that the two cases just described are characterized by 
  $\sickson(E)\neq\sickson(E')$ and $\sickson(E)=\sickson(E')$,
  respectively. For this we use the fact that
  $\F_q^\times u_1+\F_q=\F_q^\times u_2+\F_q$, or
  $\F_qu_1+\F_q=\F_qu_2+\F_q$, is equivalent to 
  $\F_q(u_1^q-u_1)=\F_q(u_2^q-u_2)$. This is an instance of the equivalence
  $\dickson(L_1)=\dickson(L_2)\iff L_1=L_2$ for lines $L_1,L_2$
  through the same point (in this case the point
  $\F_q=\F_q1$).\footnote{It is also straightforward to show directly
    that $u\mapsto(u^q-u)^{q-1}$ is a 
    separating invariant for the orbits of
    $\AGL(1,\F_q)$ on $\F_{q^4}$, i.e.\ $\F_q^\times
    u_1+\F_q=\F_q^\times u_2+\F_q$ iff 
    $(u_1^q-u_1)^{q-1}=(u_2^q-u_2)^{q-1}$.}
  Using this fact and $u^q-u=\prod_{\lambda\in\F_q}(u+\lambda)$ we can
  rewrite the collision criterion \eqref{eq:collision} as
  \begin{align*}
    g(z)^q-g(z)&=\prod_{\nu\in\F_q}\frac{\dickson(d+\mu
      c,z)+\nu\dickson(a,z)}{\dickson(a,z)}\\
    &=\dickson(a,z)^{-q}\prod_{\nu\in\F_q}\dickson(d+\mu
      c+\nu a,z)\\
      &\in\dickson(Z)^{-q}\prod_{\substack{L\subset E\\\F_qz\in
          L\wedge L\neq Z}}\dickson(L)\\
      &=\dickson(Z)^{-q-1}\prod_{\substack{L\subset E\\\F_qz\in
          L}}\dickson(L)\\
      &=\frac{\dickson(E)}{\dickson(Z)^{q+1}}\cdot(\F_qz)^{q}\\
      &=\sickson(E)\cdot(\F_qz)^{q}=\sickson(E')\cdot(\F_qz)^{q},
  \end{align*}
  where we have used that the product of all points in $E$ on the
  $q+1$ lines through $\F_qz$ involves $\F_qz$ exactly $q+1$ times and
  all other points exactly once. Cancelling the factor $(\F_qz)^{q}$
  completes the proof of the lemma.
\end{proof}

As a consequence of Lemma~\ref{lma:collision} we obtain that there exist
subsets $\mathcal{N}_1'\subseteq\mathcal{N}_1$ of size
$\#\mathcal{N}_1'=(q-1)\cdot\#\image(\sickson)$ which
can be added to the expurgated LMRD code $\mathcal{L}_0$ while still maintaining
$t=2$. For this we choose for each point $Q$ in the image of
$\sickson$ a plane $E\neq W$ with $\sickson(E)=Q$ and take
$\mathcal{N}_1'$ as the union of all sets $N(Z,P_1,\F_q^\times
g)$ parametrized by these planes. In the smallest case $q=2$, where
$\#N(Z,P_1,\F_q^\times g)=1$,
such a set $\mathcal{N}_1'$ is clearly maximal.\footnote{Whether such
  sets $\mathcal{N}_1'$ are maximal in general remains an open problem.}
% to be decided,
% but the discussion of this problem is deferred until we know about the
% values taken by $E\mapsto\sickson(E)$ and their multiplicities.}

Hence our next goal is to obtain more detailed information
on the map
$E\mapsto\sickson(E)$ with domain the set
of $q^3+q^2+q$ planes $E\neq W$ in $\PG(\F_{q^4}/\F_q)$, and in
particular determine its image size. As a first step towards this
we establish an explicit formula for $\sickson(E)$. The formula is
stated in terms of the absolute invariant
$\sickson(E)^{q-1}\in\F_{q^4}^\times$, which is obtained by composing
$E\mapsto\sickson(E)$ with the group isomorphism
$\F_{q^4}^\times/\F_q^\times\to(\F_{q^4}^\times)^{q-1}$,
$r\F_q^\times\mapsto r^{q-1}$. 

\begin{lemma}
  \label{lma:sickson}
  For a plane $E=aW\neq W$ of $\PG(\F_{q^4}/\F_q)$ we have
  \begin{equation*}
    \sickson(E)^{q-1}=1-\frac{a^{(q-1)(q^2+1)}-1}{a^{q-1}-1}.
  \end{equation*}
\end{lemma}
\begin{proof}
  First we show $\dickson(W)=\F_q\uone$ or, equivalently,
  $\dickson(W)^{q-1}=-1$, with $\uone$ as in
  Lemma~\ref{lma:dickson}. From $X^{q^3}+X^{q^2}+X^q+X=\prod_{w\in
    W}(X-w)$ the product of all elements in $W\setminus\{0\}$ is $1$.
  For a point $P=\F_qx$ the quantity $\dickson(P)^{q-1}=x^{q-1}$
  differs from $\prod_{x\in
    P\setminus\{0\}}x=\prod_{\lambda\in\F_q^\times}(\lambda
  x)%=x^{q-1}\prod_{\lambda\in\F_q^\times}\lambda
  =-x^{q-1}$ just by its sign. Hence we have
  $\dickson(W)^{q-1}=(-1)^{q^2+q+1}\prod_{w\in
    W\setminus\{0\}}w=(-1)^{q^2+q+1}=-1$ as claimed.\footnote{Note
    that the last equality is trivially true for even $q$.}

  % where $k=\dim(U)$. for any plane $E$ in
  % $\PG(\F_{q^4}/\F_q)$, is equal to the product of all powers
  % $x^{q-1}$ elements in the $\F_q$-space $E$ (and similarly for a line
  % $L$). This allows us to compute $\dickson(W)$: From
  % $X^{q^3}+X^{q^2}+X^q+X=\prod_{w\in W}(X-w)$ we have $\prod_{w\in
  %   W\setminus\{0\}}w=1$ and hence
  % $\dickson(W)^{q-1}=(-1)^{1+q+q^2}=-1$ (note that the last equality
  % is trivially true for even $q$). (the coefficient of $X$ in the
  % polynomial $X^{q^3}+X^{q^2}+X^q+X$) and hence $\dickson(W)=\F_q$.
  
  This gives $\dickson(aW)=\F_qa^{q^2+q+1}\uone$ and $\dickson(aW)^{q-1}
  =-a^{q^3-1}$ for any $a\in\F_{q^4}^\times$.
  
  Since $\sickson(aW)^{q-1}=\dickson(aW)^{q-1}/\dickson(Z)^{q^2-1}$,
  where $Z=W\cap aW$, we also need to compute $\dickson(W\cap aW)$. This
  can be done as follows:

  The $\F_q$-space $W\cap aW$ is the set of zeros of the polynomial
  \begin{align*}
    &X^{q^3}+X^{q^2}+X^q+X-a^{q^3}\left(
      (a^{-1}X)^{q^3}-(a^{-1}X)^{q^2}-(a^{-1}X)^{q}-a^{-1}X\right)\\
    &\quad=(1-a^{q^3-q^2})X^{q^2}+(1-a^{q^3-q})X^{q}+(1-a^{q^3-1})X\\
    &\quad=(1-a^{q^3-q^2})\left(X^{q^2}+\frac{1-a^{q^3-q}}{1-a^{q^3-q^2}}X^q
      +\frac{1-a^{q^3-1}}{1-a^{q^3-q^2}}X\right),
  \end{align*}
  and hence
  \begin{align*}
    \dickson(W\cap aW)^{q-1}&=\frac{1-a^{q^3-1}}{1-a^{q^3-q^2}},\\
    \sickson(aW)^{q-1}&=\frac{-a^{q^3-1}(1-a^{q^3-q^2})^{q+1}}{(1-a^{q^3-1})^{q+1}}\\
    &=-\frac{a^{q^3-1}(1-a^{1-q^3})(1-a^{q^3-q^2})}{(1-a^{1-q})(1-a^{q^3-1})}\\
    &=\frac{1-a^{q^3-q^2}}{1-a^{1-q}}=\frac{a^{q-1}-a^{q^3-q^2+q-1}}{a^{q-1}-1}\\ 
    %&=1-\frac{a^{q^3-q^2}-a^{1-q}}{1-a^{1-q}}\\
    &=1-\frac{a^{q^3-q^2+q-1}-1}{a^{q-1}-1}\\
    &=1-\frac{a^{(q-1)(q^2+1)}-1}{a^{q-1}-1},
  \end{align*}
  as asserted.
\end{proof}
From Lemma~\ref{lma:sickson} it is clear that $\sickson(E)=\F_q$ for
the planes of the form $E=a^{q+1}W\neq W$ and no other planes. Since
there are $q^2$ such planes, we have that $\#\image(\sickson)\leq
q^3+q^2+q-(q^2-1)=q^3+q+1$. It turns out that equality holds in this
bound, and hence a maximum of $\#\mathcal{N}_1'=(q-1)(q^3+q+1)$ planes
passing through any given point $P\in S$ can be added to
$\mathcal{L}_0$ without increasing $t$. Before proving this theorem,
we note that the existence of collisions already implies that a
$q$-analogue of the Fano plane cannot be constructed by our present
method.
% \footnote{A necessary condition for
%   this would have been that the collision graph contains no edges. In
%   fact the collision graph consists of $q^3+q$ isolated vertices and a
%   complete graph $\mathrm{K}_{q^2}$; cf.\ Theorem~\ref{thm:sickson}.

\begin{theorem}
  \label{thm:sickson}
  Let $\mathcal{L}_0$ be the plane subspace code of size $q^8-q^7+q^3$
  obtained from the lifted Gabidulin code $\mathcal{L}$ by removing
  all planes $\graph_f$ corresponding to binomials $f(x)=r(ux^q-u^qx)$
  with $r\in\F_{q^4}^\times$, $u\in\F_{q^4}\setminus W$. Then
  $\mathcal{L}_0$ can be augmented by $(q^4-1)(q^3+q+1)$ new planes
  meeting $S$ in a point, $(q-1)(q^3+q+1)$ of them passing through
  any point $P\in S$, to a subspace code $\mathcal{C}$ with size
  $\#\mathcal{C}=q^8+q^5+q^4-q-1$. Moreover, $\mathcal{C}$
  may be chosen as a $\Sigma$-invariant code.
\end{theorem}
\begin{proof}
  As discussed above, we need only show that the values of $\sickson$
  on the $q^3+q$ planes not of the form $a^{q+1}W$ are distinct. This
  is equivalent to
  \begin{equation}
    \label{eq:xy}
    \frac{x-1}{y-1}\neq\frac{x^{q^2+1}-1}{y^{q^2+1}-1}
  \end{equation}
  for any pair of distinct elements $x,y\in\F_{q^4}^\times$ that are
  ($q-1$)-th powers but not ($q^2+1$)-th roots of unity.

  Assume by contradiction that equality holds in \eqref{eq:xy} for
  some pair $x,y$. Then, since the right-hand side is in the subfield
  $\F_{q^2}$, we can conclude that also
  \begin{equation*}
    \frac{x-1}{y-1}=\left(\frac{x-1}{y-1}\right)^{q^2}
    =\frac{x^{q^2}-1}{y^{q^2}-1}.
  \end{equation*}
  The two equations can be rewritten as
  \begin{align*}
    \sum_{i=0}^{q^2}x^i&=\frac{x^{q^2+1}-1}{x-1}=\frac{y^{q^2+1}-1}{y-1}
    =\sum_{i=0}^{q^2}y^i,\\
    \sum_{i=0}^{q^2-1}x^i&=\frac{x^{q^2}-1}{x-1}=\frac{y^{q^2}-1}{y-1}
    =\sum_{i=0}^{q^2-1}y^i,    
  \end{align*}
  and together imply $x^{q^2}=y^{q^2}$ and hence $x=y$; contradiction.
\end{proof}
\begin{remark}
  \label{rmk:sickson}
  The map $\F_qa\mapsto\sickson(aW)$ leaves each coset of the subgroup 
  consisting of the ($q+1$)-th powers (or ($q^2+1$)-th roots of unity) in
  $\F_{q^4}^\times/\F_q^\times$ invariant and induces bijections on
  all nontrivial cosets; in particular, the set of values excluded
  from $\image(\sickson)$ consists of the $q^2$ points $\neq\F_q$ in
  $\PG(\F_{q^4}/\F_q)$ that are of the form $\F_qa^{q+1}$.

  This refinement of Theorem~\ref{thm:sickson} follows from
  \begin{align*}
    \sickson(aW)^{(q-1)(q^2+1)}&=\left(\frac{a^{q-1}-a^{(q-1)(q^2+1)}}
      {a^{q-1}-1}\right)^{q^2+1}=\left(\frac{a^{q-1}(1-a^{q^3-q^2})}
      {a^{q-1}-1}\right)^{q^2+1}\\
    &=\left(\frac{a^{q^3-q^2}(1-a^{q-1}}
      {a^{q^3-q^2}-1}\right)\left(\frac{a^{q-1}(1-a^{q^3-q^2})}
      {a^{q-1}-1}\right)\\
    &=a^{q^3-q^2+q-1}=a^{(q-1)(q^2+1)},
  \end{align*}
  which shows the claimed coset invariance, and the known
  behaviour of $\F_qa\mapsto\sickson(aW)$ on the subgroup of ($q+1$)-th
  powers and its complement. In the next section we will discuss the
  geometric significance of this subgroup.
\end{remark}

\section{Extensions}\label{sec:ext}

The subspace code $\mathcal{C}$ of Theorem~\ref{thm:sickson} is far
from being unique---we can select the $q-1$ new planes in one of the
$q^2$ ``collision classes'' independently at each point of $S$ and even
mix planes from different collision classes for $q>2$, resulting in
at least $(q^2)^{q^3+q^2+q+1}$ different choices for $\mathcal{C}$
(exactly $4^{15}$ different choices for $q=2$).

On the other hand, if we omit the selection of a collision class at
every point of $S$ then no ambiguity is introduced. The resulting
subspace code, we call it $\mathcal{C}_0$, has size
$\#\mathcal{C}_0=\#\mathcal{C}-(q^4-1)=q^8+q^5-q$ and is clearly
$\Sigma$-invariant. Moreover, the size of a
maximal\footnote{``Maximal'' refers to ``maximal size'', not the
  weaker ``maximal with respect to set inclusion''.}  extension
$\overline{\mathcal{C}_0}$ of $\mathcal{C}_0$ is no less than the size
of a maximal extension $\overline{\mathcal{C}}$ of $\mathcal{C}$.

The planes we should consider for augmenting
$\mathcal{C}_0$ are essentially of two types---at most $q^4-1$
planes meeting $S$ in a point and at most
$\gauss{4}{2}{q}=q^4+q^3+2q^2+q+1$ planes meeting $S$ in a
line.\footnote{Adding planes contained in $S$ to $\mathcal{C}_0$ is
  not an option.} Hence the size of 
$\overline{\mathcal{C}_0}$ is a priori bounded by
$q^8+q^5+q^4-q-1\leq\#\overline{\mathcal{C}_0}\leq
q^8+q^5+2q^4+q^3+2q^2$. For large $q$ one may consider
%$\#\overline\approx q^8+q^5+Cq^4$ with $1\leq C\leq 2$ 
this as a satisfactory answer to the extension problem for
$\mathcal{C}_0$, but for small values of $q$ this is certainly not true.

For more precise results we need to describe the free lines of
$\mathcal{C}_0$ meeting $S$ in a point. Prior to this description,
we collect a few geometric facts about the coset partition of $\F_{q^4}^\times$
relative to the subgroup $O$ of ($q+1$)-th powers,
% \footnote{The corresponding partition into ovoids of $\PG(3,\F_q)$
%   has been investigated in a more general setting in
%   \cite{ebert85}.}
and we prove two further auxiliary results, which seem to be of
independent interest.

The point set
$\mathcal{O}=\{\F_qa^{q+1};a\in\F_{q^4}^\times\}
=\{\F_qx;x\in\F_{q^4}^\times,x^{(q-1)(q^2+1)}=1\}$
corresponding to $O$ defines an elliptic quadric and hence an ovoid in
$\PG(\F_{q^4}/\F_q)\cong\PG(3,\F_q)$. This can be seen by rewriting
$x^{(q-1)(q^2+1)}=x^{q^3-q^2+q-1}=1$ as $x^{q^3+q}-x^{q^2+1}=0$ and
further as $\uone x^{q^3+q}-\uone x^{q^2+1}=0$, where
$\uone^{q-1}=-1$. The map $x\mapsto\uone x^{q^3+q}-\uone x^{q^2+1}$ takes
values in $\F_q$ and hence constitutes a quadratic form on
$\F_{q^4}/\F_q$. Since $\#\mathcal{O}=q^2+1$, the corresponding
quadric must be elliptic.
% \frob^3(x)\frob(x)-\frob^2(x)x
% =\frob\bigl(x\frob^2(x)\bigr)-\frob^2(x)x

Hence the coset partition with respect to $O$ determines a partition
$\mathscr{O}$ of the point set of $\PG(\F_{q^4}/\F_q)$ into $q+1$
ovoids, which are transitively permuted by $\F_{q^4}^\times$ (acting
as a Singer group).\footnote{A partition of $\PG(3,\F_q)$ into $q+1$
  ovoids is often called an \emph{ovoidal fibration}. The ovoidal
  fibration $\mathscr{O}$ has been further
  investigated in \cite{ebert85}.}

It is well-known (see e.g.\ \cite{hirschfeld98},
\cite{beutelspacher-rosenbaum92} or \cite{dembowski68}) that
$\mathcal{O}$ has a unique tangent plane in each of its points and
meets the remaining $q^3+q$ planes of $\PG(\F_{q^4}/\F_q)$ in $q+1$
points (the points of a non-generate conic). The tangent plane to
$\mathcal{O}$ in $\F_q=\F_q1$ is $W'=\uone W$ (the plane with equation
$\trace(\uone x)=0$), where as before $\uone^{q-1}=-1$. This follows
from
% for the subspace The $q^2+1$ planes of the form
% $a^{q+1}W$ (including $W$) form a dual ovoid in
% $\PG(\F_{q^4}/\F_q)\cong\PG(3,\F_q)$.  This follows from the
% (well-known?) fact that
% $\mathcal{O}=\{\F_qa^{q+1};a\in\F_{q^4}^\times\}
% =\{\F_qx;x\in\F_{q^4}^\times,x^{(q-1)(q^2+1)}=1\}$ constitutes an
% elliptic quadric in $\PG(\F_{q^4}/\F_q)$.\footnote {For a proof use
%   that $x^{(q-1)(q^2+1)}=x^{q^3-q^2+q-1}=1$ is equivalent to
%   $0=x^{q^3+q}-x^{q^2+1}=\frob^3(x)\frob(x)-\frob^2(x)x
%   =\frob\bigl(x\frob^2(x)\bigr)-\frob^2(x)x$. Hence, if
%   $u\in\F_{q^4}^\times$ satisfies $u^q=-u$ then
%   $f(x)=u\bigl(\frob^3(x)\frob(x)-\frob^2(x)x\bigr)$ has values in
%   $\F_q$ and thus defines a quadratic form on $\F_{q^4}/\F_q$ with set
%   of zeros $\mathcal{O}$. The claim then follows from the fact that
%   the elliptic quadric $\mathcal{E}$ is determined by its size
%   $\#\mathcal{E}=q^2+1$ among all quadrics in $\PG(3,\F_q)$.}  This
% implies that the plane $W$ contains a unique ``tangential point'' to
% the dual ovoid, in which all the dual non-secants are concurrent. This
% distinguished point is again $\F_qu$, where $u\in\F_{q^4}^\times$
% satisfies $u^q=-u$. For on one hand we have
$\trace(\uone\cdot 1)=\trace(\uone)=\uone-\uone+\uone-\uone=0$ and 
\begin{align*}
  \trace(\uone a^{q+1})&=\uone a^{q+1}-\uone a^{q^2+q}+\uone
                         a^{q^3+q^2}-\uone a^{1+q^3}\\
  &=\uone(a-a^{q^2})(a^q-a^{q^3})=\uone(a-a^{q^2})^{q+1},
\end{align*}
which shows that $\F_qa^{q+1}\notin \uone W$ unless
$\F_qa^{q+1}=\F_q$.
% (since $a=a^{q^2}$ implies $a^{q+1}\in\F_q^\times$).

It follows that each plane $E$ is tangent
to a unique ovoid in $\mathscr{O}$ and meets the remaining $q$
ovoids in $q+1$ points. More precisely, $E=aW$ is tangent to
$a\mathcal{O}$ in $a\uone$, as follows from
$\uone\mathcal{O}=\mathcal{O}$.\footnote{Note that
  $\F_q\uone\in\mathcal{O}$. For even $q$ this is trivial. If $q$ is odd
  and $\alpha$ is a primitive
  element of $\F_{q^4}$ then
  $\uone=\alpha^{(q^3+q^2+q+1)/2}=(\alpha^{(q^2+1)/2})^{q+1}$ satisfies
  $\uone^{q-1}=-1$ and is a ($q+1$)-th power in $\F_{q^4}^\times$.}  

In particular, $W$ itself is tangent to $\mathcal{O}$ in $\F_q\uone$,
and the points of $W$ are partitioned into the singleton
$\{\F_q\uone\}$ and $q$ ovoid sections $W\cap\alpha^i\mathcal{O}$,
$1\leq i\leq q$, of size $q+1$. 

% In other words, the $q^2$ planes $E\neq W$ of the form $E=a^{q+1}W$
% intersect $W$ in the $q^2$ lines $Z\subset W$ not passing through
% $\F_qu$.\footnote{The dual cap property implies that these $q^2$ line
% are distinct.}

Now recall from Section~\ref{sec:attempt} that
$L\mapsto\dickson(L)$ maps the pencil of all lines through $\F_qa$
bijectively onto the plane $a^{q+1}W$. The planes of this form are
exactly the tangent planes to $\mathcal{O}$ and
represent a dual ovoid $\mathcal{O}^*$ in $\PG(\F_{q^4}/\F_q)$.
Hence we can dualize each of the above properties. In particular
this gives that the $q^2$ planes in $\mathcal{O}^*\setminus\{W\}$
(i.e.\ those with $\sickson(E)=\F_q$, the ``colliding planes'')
intersect $W$ in the $q^2$ lines not passing through the
distinguished point $\F_q\uone$.\footnote{The point $\F_q\uone$ represents
  the dual tangent plane to $\mathcal{O}^*$ in $W$, and the $q^2$
  lines represent the dual lines connecting $W\in\mathcal{O}^*$ to
  the remaining points of $\mathcal{O}^*$.}

Our final and most important geometric observation relates the line
orbits of the Singer group $\F_{q^4}^\times$ to the ovoidal fibration
$\mathscr{O}$. Since $\dickson(rL)=r^{q+1}\dickson(L)$ for
$r\in\F_{q^4}^\times$, every line orbit $[L]$ corresponds to a
unique ovoid in $\mathscr{O}$ (the ovoid containing the point
$\dickson(L)$). The map $[L]\to\dickson(L)\mathcal{O}$ must be a
bijection, since this is true for $L\mapsto\dickson(L)$ at any fixed
point $\F_qa$ and every line orbit (resp., ovoid) contains a line
through $\F_qa$ (resp., has a nonempty ovoid section in $a^{q+1}W$).

In fact the foregoing shows that there are $q$ regular line orbits
$[L]$ (i.e., of length $q^3+q^2+q+1$) and one ``short'' line orbit of
length $q^2+1$ represented by the subfield $\F_{q^2}$ (since
$\dickson(\F_{q^2})=\F_q\uone$). The short orbit contains exactly one
line through each point (i.e., it forms a
line spread); any regular orbit contains $q+1$ lines through each
point $\F_qa$, which form a quadric cone with vertex $\F_qa$; in
particular no three of these $q+1$ lines are coplanar.\footnote{See
  \cite{glynn88a} for more information on this.} 

We have seen in Lemma~\ref{lma:dickson} that $a,b\in W$ implies
$\uone\dickson(a,b)\in W$ (i.e.\ $\dickson(Z)\in W'=\uone W$ for any
line $Z=\langle a,b\rangle\subset W$). The following similar but less
obvious result will be used in the sequel.

\begin{lemma}
  \label{lma:dickson2}
  For $a,b\in W$ we also have $\uone a^{q^3}\dickson(a,b)^{q+1}\in W$.
  %where as before $u^{q-1}=-1$.
\end{lemma}
This is %trivially true in the case $\F_qa=\F_qb$ and 
easily seen to be
equivalent to $z^{q^3}\dickson(Z)^{q+1}\in W'$ for all lines
$Z$ in $W$ and all points $\F_qz$ on $Z$.
\begin{proof}
  % For fixed $a\in W$ the space $L_a=\{\dickson(a,b);b\in W\}\subset W'$ 
%   represents a line of $\PG(\F_{q^4}/\F_q)$. From the theory of
%   linearized polynomials \cite{ore33a} we know that $\prod_{c\in
%     L_a}(X-c)=X^{q^2}+d_1X^q+d_2X$ for certain elements
%   $d_1,d_2\in\F_{q^4}$, and our first task is to determine $d_1,d_2$.

%   As in the proof of Lemma~\ref{lma:sickson} we have
%   \begin{equation*}
%     d_2=(-1)^2\dickson(L_a)^{q-1}
%     =\left(\prod_{a\in Z\subset W}\dickson(Z)\right)^{q-1}
%     =\bigl(a^q\dickson(W)\bigr)^{q-1}=-a^{q^2-q},
%   \end{equation*}
%   since $\dickson(W)=\F_qu$. Next observe that
%   $\dickson(a,u)=au^q-a^qu=-(a+a^q)u$. 
First note that $W$ contains a unique line $L_0=\F_{q^2}\utwo$ of the
short line orbit, which is determined by
$\utwo^{q^2-1}=-1$.\footnote{Thus $L_0$ is the $\F_{q^2}$-analogue of
  the point $\F_q\uone$ and can also be seen as the kernel of the
  relative trace map $\trace_{\F_{q^4}/\F_{q^2}}$.}
Since the map $W\to\F_{q^4}$, $b\mapsto\uone a^{q^3}\dickson(a,b)^{q+1}$ is
constant on lines through $\F_qa$, it suffices to consider the
cases (i) $\F_qa\notin L_0$, $\F_qb\in L_0$ and (ii) $\F_qa\in L_0$,
$b\in W$ arbitrary.

(i) Since all nonzero elements $b\in L_0$ satisfy $b^{q^2-1}=-1$, we
write $b=\utwo$ in this case. Our task is to show that the alternating
sum of the conjugates (over $\F_q$) of
\begin{align*}
  a^{q^3}\dickson(a,\utwo)^{q+1}
  &=a^{q^3}(a\utwo^q-a^q\utwo)(a^q\utwo^{q^2}-a^{q^2}\utwo^q)
    =-a^{q^3}(a\utwo^q-a^q\utwo)(a^q\utwo+a^{q^2}\utwo^q)\\
  &=-a^{q^3+q+1}\utwo^{q+1}+a^{q^3+2q}\utwo^2-a^{q^3+q^2+1}\utwo^{2q}
    +a^{q^3+q^2+q}\utwo^{q+1}
\end{align*}
is equal to zero. Since $a^{q^3+q+1}$ and $a^{q^3+q^2+q}$ are
conjugate and $\utwo^{q+1}=\uone$, the alternating sums of the
conjugates of the first and last summand cancel. For the two summands
in the middle we obtain likewise
\begin{align*}
  &a^{q^3+2q}\utwo^2-a^{2q^2+1}\utwo^{2q}+a^{2q^3+q}\utwo^2-a^{q^2+2}\utwo^{2q}\\
  &\qquad\quad
    -(a^{q^3+q^2+1}\utwo^{2q}-a^{q^3+q+1}\utwo^2+a^{q^2+q+1}\utwo^{2q}
    -a^{q^3+q^2+q}\utwo^2)\\
  &\quad=a^{q^3+q}(a^q+a^{q^3}+a+a^{q^2})\utwo^2-a^{q^2+1}(a^{q^2}+a+a^{q^3}+a^q)
    \utwo^{2q}=0,
\end{align*}
since $a\in W$.

(ii) Writing $a=\utwo$, we have
\begin{align*}
  \utwo^{q^3}\dickson(\utwo,b)^{q+1}&=-\utwo^q(\utwo
                                      b^q-\utwo^qb)(\utwo^q b^{q^2}+\utwo b^q)\\
  &=b^{q+1}\utwo^{2q+1}-b^{2q}\utwo^{q+2}+b^{q^2+1}\utwo^{3q}-b^{q^2+q}\utwo^{2q+1}.
\end{align*}
The alternating sum of the conjugates of the third summand is 
$b^{q^2+1}\utwo^{3q}+b^{q^3+q}\utwo^3-b^{q^2+1}\utwo^{3q}-b^{q^3+q}\utwo^3=0$. For
the alternating sum of the conjugates of the rest we obtain, using
$(\utwo^{2q+1})^q=\utwo^{2q^2+q}=\utwo^{q+2}$,
$(\utwo^{q+2})^q=\utwo^{q^2+2q}=-\utwo^{2q+1}$ and
$b^{q+1}+b^2+b^{q^3+1}=(b^q+b+b^{q^3})b=-b^{q^2+1}$, etc.,
\begin{align*}
  &(b^{q+1}-b^{q^3+q^2}-b^{2q^2}+b^2-b^{q^2+q}+b^{q^3+1})\utwo^{2q+1}\\
  &\qquad\quad
    +(-b^{q^2+q}+b^{q^3+1}-b^{2q}+b^{2q^3}+b^{q^3+q^2}-b^{q+1})\utwo^{q+2}\\
  &\quad=(-b^{q^2+1}+b^{q^2+1})\utwo^{2q+1}+(b^{q^3+q}-b^{q^3+q})\utwo^{q+2}=0.
\end{align*}
This completes the proof of the lemma.
\end{proof}
The second auxiliary result is the projective version of
Lemma~\ref{lma:sickson}.
\begin{lemma}
  \label{lma:sickson2}
  For $a\in\F_{q^4}\setminus\F_q$ we have
  $\sickson(aW)=\F_q\uone a^{-q}(a^q-a)^{q+1}$.\footnote{Note that
    $aW=W$ is equivalent to $a\in\F_q^\times$ (e.g., by Singer's Theorem).}
\end{lemma}
\begin{proof}
  By Lemma~\ref{lma:sickson},
  \begin{align*}
    \sickson(aW)^{q-1}&=\frac{a^{q-1}-a^{q^3-q^2+q-1}}{a^{q-1}-1}
    =\frac{a^q-a^{q^3-q^2+q}}{a^q-a}
    =-\frac{a^{q^3}-a^{q^2}}{a^{q^2-q}(a^q-a)}\\
    &=-\frac{(a^q-a)^{q^2-1}}{a^{q^2-q}}
      =\left(\uone\cdot\frac{(a^q-a)^{q+1}}{a^q}\right)^{q-1}.
  \end{align*}
  The result follows.
\end{proof}
Now we are ready to resume the analysis of augmenting
$\mathcal{C}_0$. Recall from Section~\ref{sec:attempt} that the
$(q-1)(q^3+q)$ planes in $\mathcal{C}_0$ meeting $S$ in
$P_1=\F_q(0,1)$ have the form
$N=N(Z,P_1,g)=\bigl\{(x,g(x)+\nu);x\in Z,\nu\in\F_q\bigr\}$, where
$Z=\langle a,b\rangle\subset W$ is $2$-dimensional,
$g(x)=\dickson(\lambda d+\mu c,x)/\dickson(a,b)$ and the plane
$E=\langle a,b,\lambda d+\mu c\rangle$ is not of the form $u^{q+1}W$.

In what follows, by a \emph{free line} we mean
a line not covered by a plane in $\mathcal{C}_0$, and by a
\emph{free plane} a plane which contains only free lines and hence
can be individually added to $\mathcal{C}_0$ without increasing $t$.
From Section~\ref{sec:attempt} we know that the $(q-1)q^2$ planes
$N(Z,P_1,g)$ with $E$ of the form $u^{q+1}W$ and their images under
$\Sigma$ are free. We will denote this set of $(q^4-1)q^2$ free planes
by $\mathcal{N}''$, so that
$\mathcal{N}=\mathcal{N}'\uplus\mathcal{N}''$ in the terminology of
Section~\ref{sec:attempt}. 

For the statement of the next lemma recall that the  
$4$-flats in $\PG(W\times\F_{q^4})$ above $S$ are of the form
$F=\F_qx\times\F_{q^4}=\F_q(x,0)+S$ with $\F_qx$ a point in $W$ (i.e.\
$x\in W$ is uniquely determined up to scalar multiples in $\F_q^\times$).

\begin{lemma}
  \label{lma:extension}
  Let $F=\F_qx\times\F_{q^4}$ be a $4$-flat containing $S$ and
  $P_0=\F_q(x,0)=F\cap W$.
  \begin{enumerate}[(i)]
  \item A line $L\subset F$ meeting $S$
    in a point is free if and only if either $P_0\in L$ or the
    plane generated by $P_0$ and $L$ meets $S$ in a line $L'$
    such that $\dickson(L')\in x^q\mathcal{O}$.
  \item A plane $E\subset F$ meeting $S$
    in a line $L'$ is free if and only if $P_0\in E$ and
    $\dickson(L')\in x^q\mathcal{O}$.
  \end{enumerate}
\end{lemma}
Note that, in view of the preceding discussion, the condition
$\dickson(L')\in x^q\mathcal{O}$ holds precisely for the lines $L'$ in
a certain line orbit of $\F_{q^4}^\times$ on
$\PG(S/\F_q)\cong\PG(\F_{q^4}/\F_q)$. Points $\F_qx$, $\F_qx'$ in
the same ovoid section $W\cap x\mathcal{O}=W\cap x'\mathcal{O}$ are
associated with the same line orbit, and the induced map from ovoid
sections to line orbits is a bijection.\footnote{The ovoid
  $x^q\mathcal{O}=(x\mathcal{O})^q$ differs from $x\mathcal{O}$ only
  by conjugation with the Frobenius automorphism of $\F_{q^4}/\F_q$. If we
  choose orbit representatives with $1\in L'$ then the condition of the lemma
  becomes $\dickson(L')\in W\cap x^q\mathcal{O}$, the conjugate
  ovoid section in $W$.} Moreover, the degenerate ovoid section
$\{\F_q\uone\}$ is associated with the short line orbit $[\F_{q^2}]$ (since
$\dickson(\F_{q^2})=\F_q\uone\in\mathcal{O}=\uone^q\mathcal{O}$).

\begin{proof}[Proof of the lemma]
  Since the sets of free lines and free planes,
  as well as the stated conditions, are $\Sigma$-invariant, it suffices
  to consider the cases $P_1\in L$ and $P_1\in E$.

  (i) The line $L_0=\langle P_0,P_1\rangle=\F_qx\times\F_q$ is free,
  since $N=N(Z,P_1,g)\in\mathcal{C}_0$ has
  $g(x)=\dickson(\lambda d+\mu c,x)/\dickson(a,x)\notin\F_q$. The
  remaining $q^3-1$ lines in $F$ meeting $S$ in $P_1$ have the form
  $L=\F_q(x,y)+P_1$ with $y\in\F_{q^4}\setminus\F_q$ and correspond to
  nontrivial additive cosets of $\F_q$ in $\F_{q^4}$. Inspecting the
  proof of Lemma~\ref{lma:collision} shows that such a line $L$ is
  free iff $\F_q(y^q-y)=\F_q\dickson(1,y)\neq x^q\sickson(E)$ for all
  planes $E$ through $\F_qx$ with $E\notin\mathcal{O}^*$. Since this
  condition depends only on $\F_q^\times y$, the
  free lines form a union of planes through $L_0$ whose intersecting
  lines $L'=\langle 1,y\rangle$ with $S$ are determined by the
  conditions $\dickson(L')\neq x^q\sickson(E)$.\footnote{The points in
    $\langle L_0,L'\rangle$ on the $q-1$ lines through $P_1$ different
    from $L_0,L'$ are those of the form $\F_q(x,y')$ with $y'$ in the
    $\AGL(1,\F_q)$-orbit $\F_q^\times y+\F_q$.}
  
  The planes $E=uW$ containing $\F_qx$ are characterized by $x/u\in
  W$. One such plane is $W$, which will be excluded from now on.
  Using Lemma~\ref{lma:sickson2}, homogeneity of $\dickson$ and Lagrange's
  Theorem for the group $\F_{q^4}^\times/\F_q$, we obtain
  \begin{align*}
    x^q\sickson(uW)&=\F_q\uone(x/u)^q\dickson(1,u)^{q+1}
    =\F_q\uone(u/x)^{q^2+q+1}\dickson(x/u,x)^{q+1}\\
    &=\F_q\uone(x/u)^{q^3}\dickson(x/u,x)^{q+1}.
  \end{align*}
  By Lemma~\ref{lma:dickson2}, $x^q\sickson(uW)\in W$ for all planes
  $uW\neq W$ containing $\F_qx$.  Now we
  distinguish two cases.

  \textsl{Case~1}: $\F_qx=\F_q\uone$. In this case, since no plane in
  $\mathcal{O}^*$ except $W$ passes through $\F_q\uone$, all $q^2+q$
  planes $uW\neq W$ containing $\F_q\uone$ provide a condition
  $\dickson(L')\neq \uone^q\sickson(uW)$. But $\dickson(L')\in W$ and
  the invariants $\sickson(uW)$ are distinct and $\neq 1$. Hence
  $\dickson(L')=\F_q\uone^q=\F_q\uone$ remains as the only
  possibility. This implies $L'=\F_{q^2}$ and
  $\dickson(L')\in\uone\mathcal{O}=\mathcal{O}$ as asserted.

  \textsl{Case~2}: $\F_qx\neq\F_q\uone$. In this case exactly
  $q$ of the planes in $\mathcal{O}^*$ pass through $\F_qx$ and the
  condition $\dickson(L')\neq x^q\sickson(uW)$ applies to $q^2$
  planes. Since $(x/u)^{q^3}\in x^{q^3}\mathcal{O}$ iff
  $u^{q^3}\in\mathcal{O}$ iff $u\in\mathcal{O}$, we must have
  $x^q\sickson(uW)\notin x^{q^3}\mathcal{O}$ for these $q^2$
  planes. Hence the $q^2$ values taken by $x^q\sickson(uW)$ form the
  complementary set $W\setminus x^{q^3}\mathcal{O}$ and
  the condition reduces to $\dickson(L')\in x^{q^3}\mathcal{O}$. Since
  $x^{q^3-q}=(x^{q^2-q})^{q+1}\in\mathcal{O}$, this is in turn
  equivalent to $\dickson(L')\in x^q\mathcal{O}$ as asserted.

  (ii) Clearly any plane satisfying these conditions is free.
  Conversely, if $E$ is free and $P_0\in E$ then Part~(i) can be applied to
  any line $L\subset E$ satisfying $P_0\notin L\neq L'$ and gives
  $\dickson(L')\in x^q\mathcal{O}$. Thus it remains to show that
  in the case $P_0\notin E$ the plane $E$ cannot be free.

  Consider the solid $T=\langle E,P_0\rangle$, which meets $S$ in a
  plane $E'\supset L'$. Connecting $P_0$ to the $q^2+q$ lines $\neq
  L'$ in $E$ and applying Part~(i) gives that all lines $\neq L'$ in
  $E'$ must be in the same line orbit of $\F_{q^4}^\times$ in
  $\PG(S/\F_q)$. Since $E'$ contains no more than $q+1$
  lines of any line orbit,\footnote{More precisely, the lines in $E'$
    fall into $q+1$ orbits---a single line in the short orbit and
    $q+1$ lines forming a dual conic in each of the $q$ regular
    orbits.} we have a contradiction, and the proof of
  the lemma is complete.
\end{proof}
In the sequel we write $\mathcal{E}$ for the set of free planes
meeting $S$ in a line. Part~(ii) of Lemma~\ref{lma:extension} says that
the planes in $\mathcal{E}$ have the form $\F_qx\times L'$
(``decomposable'' planes) with $L'$ in the line orbit associated to $\F_qx$. 

Clearly the largest extension (still having $t=2$) of $\mathcal{C}_0$
by planes in $\mathcal{E}$ is obtained in the following way: (i) Add
all $q^2+1$ planes generated by $\F_q(\uone,0)$ and a line in the
short line orbit of $\F_{q^4}^\times$ on $\PG(S/\F_q)$. These planes
have the form $\F_q\uone\times(\F_{q^2})r$ with
$r\in\F_{q^4}^\times/\F_{q^2}^\times$. (ii) For each ovoid section
$W\cap x\mathcal{O}$ of size $q+1$ decompose the associated regular
line orbit $[L']$ of $\F_{q^4}^\times$ on $\PG(S/\F_q)$ into $q+1$
mutually disjoint partial spreads and a remainder of minimum size
(i.e., the union of the partial spreads, a subset of $[L']$, should
have maximum size). Choose a bijection from $W\cap x\mathcal{O}$ to
the set of these partial spreads and add all planes $\F_qx'\times L$
with $\F_qx'\in W\cap x\mathcal{O}$ and $L$ a line in the partial
spread corresponding to $\F_qx'$.
%\footnote{As mentioned earlier,
%  $W\times\{0\}$ is identified with $W$ in the obvious way.}  
%These planes are also decomposable in the form $\F_qx'\times L$, where $L$
%is a line in $\PG(\F_{q^4}/\F_q)$.

For small values of $q$ it turns out that the regular Singer line
orbits of $\PG(3,\F_q)$ admit decompositions into fairly large partial
spreads. As a consequence, maximal extensions of $\mathcal{C}_0$ by
planes in $\mathcal{E}$ improve on the code $\mathcal{C}$ of
Theorem~\ref{thm:sickson}. Below we will discuss in more detail the cases
$q=2,3$, where the number of additional planes is $29$ and $114$
respectively.\footnote{Compare this with the number $q^4-1=15$ resp.\
  $80$ of planes in $\mathcal{N}''$ that can be added to
  $\mathcal{C}_0$, and also with the theoretical maximum of
  $\gauss{4}{2}{q}=35$ resp.\ $130$ additional planes in $\mathcal{E}$.}

Of course we are ultimately interested in finding the largest
extension of $\mathcal{C}_0$ by planes of any of the two types. For
$q=2$ it turns out that all but one of the theoretical maximum of
$15+29=44$ additional planes can be added to $\mathcal{C}_0$,
resulting in the largest presently known subspace code
$\overline{\mathcal{C}_0}$ of size $286+43=329$; cf.\
\cite{braun-reichelt14,lt:0328}. This case will be considered further
below, culminating in a computer-free construction of one such
code.

It seems difficult, however, to generalize the analysis in the binary
case to larger values of $q$. In the ternary case $q=3$ the largest
extension of $\mathcal{C}_0$ we have found by a computer search has size
$\#\overline{\mathcal{C}_0}=6977$,\footnote{Compare this with the upper
  bound $\#\overline{\mathcal{C}_0}\leq 6801+80+114=6995$.}
but we do not yet know whether this is the true maximum.

We summarize our present knowledge about the extension problem for
$\mathcal{C}_0$ in the following theorem. Part~(i) and~(ii) are the
result of a computer search. For the computation of canonical forms
and automorphism groups of subspace codes, the algorithm
in~\cite{feulner13} (based on~\cite{feulner09}, see
also~\cite{feulner-diss}) is used.  Part~(iii) represents a slight
improvement of Theorem~\ref{thm:sickson} for general $q$.

\begin{theorem}
  \label{thm:final}
  Let $\mathcal{C}_0$ be the plane subspace code of size $q^8+q^5-q$
  obtained by the expurgation-augmentation process described in
  Section~\ref{sec:attempt} and with no planes in $\mathcal{N}''$
  selected.
  \begin{enumerate}[(i)]
  \item For $q=2$ maximal extensions
    $\overline{\mathcal{C}_0}$ of $\mathcal{C}_0$ have size
    $\#\overline{\mathcal{C}_0}=329$. There exist $26\,496$ different
    isomorphism types of such extensions, all with trivial
    automorphism group. Moreover, both possible intersection patterns
    with $S$, viz.\ $(a_0,a_1,a_2,a_3)=(136,164,29,0)$ and
    $(136,165,28,0)$, occur
    with numbers of isomorphism types $10\,368$ and $16\,128$, respectively.
  \item For $q=3$ there exists an extension
    $\overline{\mathcal{C}_0}$ of
    size $\#\overline{\mathcal{C}_0}=6977$.
  \item For general $q$ there exists an extension
    $\overline{\mathcal{C}_0}$ of size
    $\#\overline{\mathcal{C}_0}=q^8+q^5+q^4+q^2-q$, obtained by adding
    to the subspace code $\mathcal{C}$ of Theorem~\ref{thm:sickson}
    the $q^2+1$ planes in $\mathcal{E}$ 
    of the form $\F_q\uone\times(\F_{q^2})r$, $r\in\F_{q^4}^\times/\F_{q^2}^\times$.
  \end{enumerate}
\end{theorem}
\begin{proof}[Proof of Part~(iii)]
  Since $W$ is the tangent plane to $\mathcal{O}$ in $\F_q\uone$, $W$
  meets the remaining $q^2$ tangent planes in $\mathcal{O}^*$ in the
  $q^2$ lines not passing through $\F_q\uone$. This means that the
  planes in $\mathcal{N}''$ have the form $N(Z,P,g)$ with $Z$ not
  passing through $\F_q\uone$ and hence do not interfere with the
  $q^2+1$ new planes, which have the form
  $E=E\bigl(\F_q\uone,(\F_{q^2})r,0\bigr)$.\footnote{Viewed
    geometrically, the planes in $\mathcal{N}''$
    contain no points in $(\F_q\uone\times\F_{16})\setminus S$ and hence
    cannot have a line with a plane $\F_q\uone\times(\F_{q^2})r$ in
    common.}
\end{proof}

In the remainder of this section we will present a computer-free
construction of a maximal extension $\overline{\mathcal{C}_0}$ in the
case $q=2$ and briefly comment on the case $q=3$, which is remarkable
in several respects.\footnote{Verehrter Jubilar, Sie haben sicher
  schon bemerkt, dass die Ordnung der multiplikativen Gruppe
  $\F_{3^4}^\times$ gerade $80$ ist.}

Representing $\F_{16}$ as $\F_2[\alpha]$ with $\alpha^4+\alpha+1=0$,
we have $\F_{16}^\times=\langle\alpha\rangle$ and
$W=\{1,\alpha,\alpha^2,\alpha^4,\alpha^5,\alpha^8,\alpha^{10}\}$.\footnote{We
  can represent the points of $\PG(3,\F_2)$ by the nonzero
  elements of $\F_{16}^\times$.} The subfield $\F_4\subset\F_{16}$
represents a line of $\PG(3,\F_2)$ and generates the short line orbit
$[\F_4]=\{\F_4\alpha^{3i};0\leq i\leq 4\}$. In addition there are two
regular line orbits represented by $L_1=\{1,\alpha,\alpha^4\}$ and
$L_2=\{1,\alpha^2,\alpha^8\}=\frob(L_1)$. The remaining lines through
$1$ are $\{1,\alpha^3,\alpha^{14}\}$, $\{1,\alpha^{11},\alpha^{12}\}$
in $[L_1]$ and $\{1,\alpha^6,\alpha^{13}\}$, $\{1,\alpha^7,\alpha^9\}$
in $[L_2]$. The ovoidal fibration is
$\mathscr{O}=\{\mathcal{O},\alpha\mathcal{O},\alpha^2\mathcal{O}\}$
with $\mathcal{O}=\{\alpha^{3i};0\leq i\leq 4\}$, and the
corresponding $W$-sections are $\mathcal{O}\cap W=\{1\}$,
$\alpha\mathcal{O}\cap W=\{\alpha,\alpha^4,\alpha^{10}\}$,
$\alpha^2\mathcal{O}\cap W=\{\alpha^2,\alpha^5,\alpha^8\}$.

Decomposing $[L_1]$, $[L_2]$ into partial spreads is best done in a
graph-theoretic setting. We view the lines in each orbit as vertices
of a circulant graph via $\alpha^iL\mapsto i\in\Z_{15}$. Then
$\alpha^iL\cap\alpha^jL\neq\emptyset$ iff
$j-i\in\{\pm 1,\pm 3,\pm 4\}$ for $L\in[L_1]$, and similarly for
$[L_2]$.\footnote{In general the circulant graph associated with a
  regular line orbit has as its connection set the pairwise
  differences of the logs in $\F_{q^4}^\times/\F_q^\times$ of the
  points on a representative line.} In this way partial spreads in the
line orbits correspond to cocliques (independent sets) of the
associated circulant graph, and an optimal decomposition into $q+1=3$
partial spreads corresponds to a $3$-colorable (vertex) subgraph of
maximum size.  In the case under consideration the two graphs are
isomorphic (since the orbits are interchanged by $\frob$) and have
chromatic number $4$.  It is readily seen that the maximum cocliques
in $\Gamma_1$ (the graph corresponding to $[L_1]$) are $S=\{0,2,7,9\}$
and its cyclic shifts modulo $15$, and that $\{S,S+1,S+4\}$ forms an
optimal decomposition of $[L_1]$ into $3$ partial spreads of size $4$
(and some remainder of size $3$).

At this point we see that $\mathcal{C}_0$ can be extended by
$29=5+4+4+4+4+4+4$ planes meeting $S$ in a line. In what follows, we
choose the corresponding partial line spreads as the short line orbit
(a total spread) for $F=\F_2\uone\times\F_{16}=\F_2\times\F_{16}$ and the six
partial spreads corresponding to $S$, $S+1$, $S+4$ and their images
under $\frob$.\footnote{This choice
is closely related to the essentially unique packing of the $35$ lines
of $\PG(3,\F_2)$ into $7$ spreads, which represents
a solution to Kirkman's Schoolgirl Problem. The packing is obtained by
applying a certain cyclic shift modulo $15$ to the second orbit decomposition
and then adding the $3$ lines omitted from each orbit decomposition to
the partial spreads in the other set, one at a time.}

We have yet at our disposal the actual ``wiring'' between the six
points $x\in W\setminus\{1\}$ and the six partial line
spreads. Since
$\dickson(L_1)=1\cdot\alpha\cdot\alpha^4=\alpha^5\in\alpha^2\mathcal{O}$,
Lemma~\ref{lma:extension} only stipulates that the points in
$W\cap\alpha\mathcal{O}=\{\alpha^1,\alpha^4,\alpha^{10}\}$ are
connected to the three line spreads in $[L_1]$ and the points in
$\{\alpha^2,\alpha^5,\alpha^{8}\}$ to the three line spreads in
$[L_2]$. The actual choice of the bijections 
(out of six feasible choices for each of
the two ovoid sections) should maximize the number of planes in
$\mathcal{N}''$ that can be added to further extend the
resulting subspace code of size $286+29=315$.

In order to solve this problem, we must take a closer look at the
lines covered by the planes in $\mathcal{N}''$
and how these relate to the lines covered by
the extended code of size $315$. The ``local'' situation at
$P_1=\F_2(0,1)$ is depicted in the following table:

\begin{equation*}
  \begin{array}{|c||c|c|c|c|c|c|c|}
    \hline x\backslash L'&5,10&1,4&2,8&3,14&6,13&11,12&7,9\\
    \hline\hline
    \hline0&&\times&\times&\times&\times&\times&\times\\
    \hline
    5&\times&\times&\text{c}&\times&&\times&\\
    10&\times&\text{c}&\times&&\times&&\times\\
    \hline
    1&\times&&\times&\text{c}&\times&&\times\\
    2&\times&\times&&\times&\text{c}&\times&\\
    4&\times&&\times&&\times&\text{c}&\times\\
    8&\times&\times&&\times&&\times&\text{c}\\
    \hline
  \end{array}
\end{equation*}
The rows of the table are indexed with the logs of the elements
$x\in W$ (corresponding to the $4$-flats $F$ above $W$), the columns
with pairs $(i,j)$ such that $\alpha^i+\alpha^j=1$ (corresponding to
the lines $L'$ in $\PG(\F_{16}/\F_2)$ through $1$), and the table
entries '$\times$', 'c' indicate that the plane $\F_2x\times L'$
conflicts with a plane in $\mathcal{C}_0$ (i.e.\
$\F_2x\times L'\notin\mathcal{E}$), respectively, with the two planes
$N=N(Z,P_1,g)\in\mathcal{N}''$ that have $x\in Z$. For this recall
that in general the $q$ planes in $\mathcal{N}''$ of the form
$N(Z,P_1,g)$ with $x\in Z$ cover the same set of $q-1$ lines meeting
$S$ in a point, and that these lines are in the plane $\F_qx\times L'$
with $L'$ determined by $\dickson(L')=\F_qx^q$.\footnote{The planes
  $E\in\mathcal{O}^*\setminus\{W\}$ parametrizing the planes in
  $\mathcal{N}''$ have $\sickson(E)=\F_q$, whence
  $\dickson(L')=\sickson(E)x^q=\F_qx^q$; cf.\ the proof of
  Lemma~\ref{lma:collision}.}

Now suppose we connect $x\in\{\alpha^1,\alpha^4,\alpha^{10}\}$ to one of
the three partial line spreads in $[L_1]$, say $\mathcal{S}$. Then,
writing $P_r=\F_2(0,r)$ and 
using the action of $\Sigma$ on $\mathcal{N}''$, we see that the planes
$N=N(Z,P_r,g)\in\mathcal{N}''$ with $x\in Z$ conflict with
$\F_2x\times(rL')$, where $L'$ is the line through $1$ matched to
$x$ by the 'c' entries in the table. Thus there
are precisely $4$ values of $r$ for which the later choice of a plane
$N=N(Z,P_r,g)\in\mathcal{N}''$ with $x\in Z$ is forbidden, viz.\
those $r$ for which $rL'\in\mathcal{S}$.\footnote{Since $L'\in[L_1]$
  and $[L_1]$ is regular, the correspondence $r\mapsto rL'$ is a
  bijection.}
Applying the same reasoning to all $x\in W\setminus\{1\}$ and all
valid choices for $\mathcal{S}$, we obtain the following $3\times 3$
arrays of forbidden values for $r$. As before, elements of
$\F_{16}^\times$ are represented by their logs with respect to
$\alpha$, and in place of the partial spreads we have listed the
corresponding cocliques of the circulant graph.\footnote{Thus, for example,
  $0,2,7,9$ refers to the partial spread
  $\mathcal{S}=\{\alpha^0L_1,\alpha^2L_1,\alpha^{7}L_1,\alpha^9L_1\}$
  with lines
  $\alpha^0L_1=\{1,\alpha,\alpha^4\}$,
  $\alpha^2L_1=\{\alpha^2,\alpha^3,\alpha^6\}$, 
  $\alpha^7L_1=\{\alpha^7,\alpha^8,\alpha^{11}\}$,
  $\alpha^9L_1=\{\alpha^9,\alpha^{10},\alpha^{13}\}$, and $0,4,14,3$
  to the partial spread $\mathcal{S}'=
  \{\alpha^0L_2,\alpha^4L_2,\alpha^{14}L_2,\alpha^3L_2\}=\frob(\mathcal{S})$.}
Further, the ordering of $W\setminus\{1\}$ is chosen in such a way
that the arrays are symmetric with respect to the main
diagonal.\footnote{This can be done, since the offsets of the cocliques
  are the same as that of the lines through $1$.}
  \begin{equation*}
    \begin{array}{|c||l|l|l|}
      \hline
      x\backslash\mathcal{S}&0,2,7,9&1,3,8,10&4,6,11,13\\\hline\hline
      10&0,2,7,9&1,3,8,10&4,6,11,13\\\hline
      1&1,3,8,10&2,4,9,11&5,7,12,14\\\hline
      4&4,6,11,13&5,7,12,14&8,10,0,2\\\hline
    \end{array}\qquad
    \begin{array}{|c||l|l|l|}
      \hline
      x\backslash\mathcal{S}&0,4,14,3&2,6,1,5&8,12,7,11\\\hline\hline
      5&0,4,14,3&2,6,1,5&8,12,7,11\\\hline
      2&2,6,1,5&4,8,3,7&10,14,9,13\\\hline
      8&8,12,7,11&10,14,9,13&1,5,0,4\\\hline
    \end{array}
  \end{equation*}
The task is now to match, for each of the two tables,
the row labels to the column labels in such a way that the number of
points $P_r$ that admit a non-conflicting
choice $N=N(Z,P_r,g)$, i.e.\ a choice of $Z$ such that $r$ is is
forbidden for no $x\in Z$, is maximized.

A moments reflection shows that the best we can do is to use the main
diagonals of the tables (or one of the other two row-and-column
transversals without repeated $4$-tuples) for the matching, i.e.\
$10\mapsto\{0,2,7,9\}$, $1\mapsto\{1,3,8,10\}$,
$4\mapsto\{4,6,11,13\}$, and similarly for the
second table. This ensures that for each $P_r$ at most two points
$x_1,x_2\in W\setminus\{1\}$ are forbidden and leads to a valid choice
for $Z$ unless the line through $x_1$, $x_2$ contains
$1$.\footnote{This is the only way to block all four lines
  $Z\subset W$ (the passants to $1$) by a $2$-set.} Since the only
such line is $\{1,\alpha^5,\alpha^{10}\}$ and the three $4$-tuples in
the first row of the first table are transversal to the corresponding
$4$-tuples of the second table, we can make a non-conflicting choice
of $Z$ for all but one $P_r$. When using the two main diagonals the
``bad'' point is $P_{11}$.

In all, we can extend $\mathcal{C}_0$ by $29+14=43$ planes to a
subspace code $\overline{\mathcal{C}_0}$ of size 329 as claimed.

Finally, we consider briefly the case $q=3$. Here the number of points
and lines in $S$ are $40$ and $130$, respectively, with line orbit
sizes $10$, $40$, $40$, $40$. Representing $\F_{81}$ as
$\F_3[\alpha]$ with $\alpha^4-\alpha^3-1=0$ (a generator of
$\F_{81}^\times$) and the points of
$\PG(\F_{81}/\F_3)$ as $\alpha^i$, $0\leq i<40$, we obtain
\[W=\{\alpha^5,\alpha^{13},\alpha^{15}, \alpha^{20}, \alpha^{22},
\alpha^{25}, \alpha^{26}, \alpha^{31}, \alpha^{34}, \alpha^{35},
\alpha^{37}, \alpha^{38}, \alpha^{39}\}\] with ovoid sections
$W\cap\mathcal{O}=\{\alpha^{20}\}$,
$W\cap\alpha\mathcal{O}=\{\alpha^5,\alpha^{13},\alpha^{25},\alpha^{37}\}$, 
  $W\cap\alpha^2\mathcal{O}=\{\alpha^{22},\alpha^{26},\alpha^{34},\alpha^{38}\}$,
$W\cap\alpha^3\mathcal{O}=\{\alpha^{15},\alpha^{31},\alpha^{35},\alpha^{39}\}$
  and corresponding line orbit representatives
  \begin{align*}
    L_0&=\F_9=\{\alpha^0,\alpha^{10},\alpha^{20},\alpha^{30}\},\\
    L_1&=\{\alpha^0,\alpha^{2},\alpha^{18},\alpha^{25}\},\\
    L_2&=\{\alpha^0,\alpha^1,\alpha^{28},\alpha^{37}\},\\
    L_3&=\{\alpha^0,\alpha^5,\alpha^{11},\alpha^{19}\}.
  \end{align*}
The orbit $[L_2]$ is $\frob$-invariant and admits a decomposition into
$4$ spreads, corresponding to the cocliques
$S=\{0,2,8,10,16,18,24,26,32,34\}$, $S+1$, $S+4$, $S+5$.\footnote{This
  follows from the fact that the differences $0,\pm 2\pmod{8}$ do not
  occur within the connection set $\{\pm 1,\pm 28,\pm 37,\pm 27,\pm
  36,\pm 9\}$ of the circulant graph.} 
The other two
regular line orbits $L_1$, $L_3$ are interchanged by $\frob$ and admit
an (optimal) decomposition into $5$ partial spreads of size $8$. For
$[L_1]$ the corresponding cocliques are $T=\{1,2,11,12,21,22,31,32\}$,
$T+2$, $T+4$, $T+6$, $T+8$.\footnote{Similarly due to the fact that $0,\pm
  1\pmod{10}$ do not occur within the connection set $\{\pm 2,\pm
  18,\pm 25,\pm 16,\pm 23,\pm 7\}$}
From this it follows that $\mathcal{C}_0$, of size
$\#\mathcal{C}_0=6801$, can be extended by $10+4\times 10+4\times
8+4\times 8=114$ planes in $\mathcal{E}$ to a subspace code of size
$6915$.

Proceeding further as in the case $q=2$, we find that the $4\times 4$
arrays corresponding to $[L_1]$ and $[L_3]$ do not contain
row-and-column transversals with all four $8$-tuples distinct. Thus
the argument used in the case $q=2$ to extend the intermediate
subspace code further by planes in $\mathcal{N}''$ breaks down and the
situation becomes considerably more involved. We have conducted a
non-exhaustive computer search for maximal extensions of
$\mathcal{C}_0$ (a more general approach than only trying to further
extend one particular extension of size 6915). As already mentioned,
the largest extension found in this way has size
$\#\overline{\mathcal{C}_0}=6977$.

\section{Conclusion}\label{sec:conclsuion}

We have developed the expurgation-augmentation approach to the
construction of good subspace codes, originally presented in
\cite{smt:fq11proc} and later extended in \cite{lt:0328}, in greater
depth, providing an explicit formula (in terms of the
$\sigma$-invariant) for the number of new planes meeting the special
solid $S$ in a point that can be added to the expurgated lifted
Gabidulin code without introducing a multiple cover of some line, and
a much refined analysis of the final extension step by planes meeting
$S$ in a line.

The existence problem for $q$-analogues of the Fano plane, which
provided a great deal of motivation for the present work, remains
grossly open, but this will not discourage us, nor should it discourage
anybody else in the audience, from further attempts to resolve it---at
least in the case $q=2$, for which by Moore's Law a computer attack 
will become feasible in the not too distant future.

Should a $q$-analogue indeed exist, it may be possible to construct it
using a variant of our approach, starting with either a non-Gabidulin
MRD code or a smaller set of $3\times 4$ matrices at pairwise rank
distance $\geq 2$ that cannot be embedded into an MRD
code.\footnote{When finishing up our work on plane subspace
  codes in $\PG(5,\F_q)$, we have discovered that one of the five
  isomorphism types of optimal binary subspace codes of size $77$ can
  be constructed from a set of $48$ binary $3\times 3$ matrices that is not
  extendable to an MRD code.}

The present work may also be continued by investigating, for general
$q$, the sizes of
optimal decompositions of Singer line orbits of $\PG(3,\F_q)$ into
$q+1$ partial spreads and how these should be wired to the points of
the corresponding ovoid sections in $W$ in order to maximize further
extendability by planes in $\mathcal{N}''$; cf.\ the end of
Section~\ref{sec:ext}. This should narrow down the gap between
the lower and upper bound for the size of a maximal extension
$\overline{\mathcal{C}_0}$ given at the beginning of
Section~\ref{sec:ext}; cf.\ also Theorem~\ref{thm:final}\,(iii).

Finally we believe that large portions of the machinery developed can
be generalized to subspace codes of packet lengths $v>7$. While for
larger $v$ there is no analogue of the trace-zero subspace $W$ and
hence no canonical choice for the ambient space and its corresponding
$\sigma$-invariant, it should still be possible to derive by our
method some explicit results on the number of planes in $\mathcal{N}'$ that can
be added to the expurgated subspace code, and to carry over the extension
analysis in Section~\ref{sec:ext} to some extent.

\nocite{kiermaier-laue15}

% \bibliographystyle{abbrv}
% \bibliography{strings,th,mathe}
% \end{document}

\def\cprime{$'$}

\end{document}